\documentclass[10.5pt, a4paper]{amsart}
\usepackage{amssymb, amsmath, amscd, bm, amsthm, pdfpages, setspace, hyperref, mathtools, subcaption, graphicx, xfrac, empheq, enumitem,xpatch}

\usepackage[
backend=biber,
style=numeric-comp,
sortcites=true,
maxnames=99
]{biblatex}

\DeclareFieldFormat{pagetotal}{#1 pp}
\renewbibmacro*{note+pages}{
	\printfield{note}
	\setunit{\bibpagespunct}
	\printfield{pages}
	\setunit{\addcomma\space}
	\printfield{pagetotal}
	\clearfield{pagetotal}
}
\renewbibmacro*{chapter+pages}{
	\printfield{chapter}
	\setunit{\bibpagespunct}
	\printfield{pages}
	\setunit{\addcomma\space}
	\printfield{pagetotal}
	\clearfield{pagetotal}
}

\def\@Rref#1{\hbox{\rm \ref{#1}}}
\def\Rref#1{\@Rref{#1}}
\theoremstyle{plain}
\newtheorem{theorem}{Theorem}[section]
\newtheorem{proposition}[theorem]{Proposition}

\newtheorem{lemma}[theorem]{Lemma}
\newtheorem{corollary}[theorem]{Corollary}

\theoremstyle{definition}
\newtheorem{definition}{Definition}[section]
\newtheorem{example}[definition]{Example}
\newtheorem{remark}[definition]{Remark}

\newcommand{\dist}{\mathop{\rm dist}}

\newcommand{\supp}{\mathop{\rm supp}}

\newcommand{\re}{\mathop{\rm Re}\nolimits}
\newcommand{\im}{\mathop{\rm Im}\nolimits}

\begin{document}

\title[Semigroup Growth and Resolvent Decay
for Differentiable Semigroups]{Relation between Semigroup Growth and Resolvent Decay
	for Immediately Differentiable Semigroups}

\thispagestyle{plain}

\author{Masashi Wakaiki}
\address{Graduate School of System Informatics, Kobe University, Nada, Kobe, Hyogo 657-8501, Japan}
 \email{wakaiki@ruby.kobe-u.ac.jp}

\begin{abstract}
We study the rate of growth of $\|AT(t)\|$ as $t \downarrow 0$
for an immediately differentiable $C_0$-semigroup $(T(t))_{t \geq 0}$
with generator $A$.
We assume that the resolvent of the semigroup generator 
decays on the imaginary axis
at rates described by 
functions of positive increase, which enable estimates on
scales finer than polynomial ones.
First, 
we present lower and upper bounds for the rates of growth of 
Banach space semigroups.
Next, we improve the upper estimate for Hilbert space semigroups.
Finally, for semigroups of normal operators on Hilbert spaces and 
multiplication $C_0$-semigroups on $L^p$-spaces, 
we establish an estimate that exactly captures the asymptotic behavior 
of $\|AT(t)\|$ as $t \downarrow 0$.
\end{abstract}

\subjclass[2020]{Primary 47D06; Secondary 34G10 $\cdot$ 35B65 $\cdot$ 26A12}
\keywords{$C_0$-semigroup, Immediate differentiability, 
	Rate of growth, 
	Resolvent.} 

\maketitle

	\section{Introduction}
Let $(T(t))_{t \geq 0}$ be a $C_0$-semigroup on a Banach space
$X$ with generator $A$.
We say that $(T(t))_{t \geq 0}$ is {\em immediately differentiable}
if the orbit map $t \mapsto T(t)x$ is differentiable on
the interval $(0,\infty)$ for each $x \in X$.
Holomorphic semigroups constitute
an important subclass of immediately differentiable 
semigroups.
It is well known  that 
an immediately differentiable $C_0$-semigroup
$(T(t))_{t \geq 0}$ is holomorphic if and only if
one of the following equivalent 
conditions holds (see, e.g., 
\cite[Theorem~2.5.2]{Pazy1983}, 
\cite[Theorem~II.4.6]{Engel2000}, and
\cite[Corollary~3.7.18 and Theorem~3.7.19]{Arendt2001}):
\begin{enumerate}
	\renewcommand{\labelenumi}{(\roman{enumi})}
	\item 
	$\displaystyle 
	\|AT(t)\| = O\left( \frac{1}{t} \right)$ as $t \downarrow 0$.
	\item $\displaystyle\|(i\eta- A)^{-1}\| = O\left( \frac{1}{|\eta|} \right)$ as $|\eta| \to \infty$.
\end{enumerate}
Here $O(\cdot)$ denotes
the big $O$ notation, and its definition can be found at the end of this section.
In this paper, we investigate the relation between
the growth of $\|AT(t)\|$ as $t \downarrow 0$ and 
the decay of $\|(i\eta- A)^{-1}\|$ as $|\eta| \to \infty$
for non-holomorphic $C_0$-semigroups.

Yosida \cite{Yosida1958} (see also \cite{Yosida1959}) 
considered an immediately
differentiable $C_0$-semigroup $(T(t))_{t \geq 0}$
satisfying
\[
\|AT(t)\| = O(e^{\gamma/t}) \qquad \text{as $t \downarrow 0$}
\]
for some constant $\gamma >0$.
For such $C_0$-semigroups, one has
\[
\|(i\eta- A)^{-1}\| = O
\left( 
\frac{1}{\log |\eta|}
\right) \qquad \text{as $|\eta| \to \infty$};
\]
see also \cite[p.~42]{Batty2007}.
Let $0 < \alpha  \leq 1$ for polynomial estimates.
If an immediately
differentiable $C_0$-semigroup $(T(t))_{t \geq 0}$
satisfies 
\begin{equation}
	\label{eq:AT_growth_poly_intro}
	\|AT(t)\| = O\left( \frac{1}{t^{1/\alpha}} \right) \qquad \text{as $t \downarrow 0$},
\end{equation}
then 
\begin{equation}
	\label{eq:resol_decay_poly_intro}
	\|(i\eta- A)^{-1}\|  = O\left( \frac{1}{|\eta|^{\alpha}} \right)  \qquad \text{as $|\eta| \to \infty$};
\end{equation}
see \cite{Crandall1969}. Conversely, it was also proved in \cite{Crandall1969} that if 
the generator $A$ of a $C_0$-semigroup
$(T(t))_{t \geq 0}$
satisfies \eqref{eq:resol_decay_poly_intro}, then
$(T(t))_{t \geq 0}$ is immediately differentiable and 
satisfies
\begin{equation}
	\label{eq:CP_intro}
	\|AT(t)\| = O\left( \frac{1}{t^{2/\alpha - 1}} \right)  \qquad \text{as $t \downarrow 0$}.
\end{equation}
This semigroup estimate was improved to 
\begin{equation}
	\label{eq:Eber_intro}
	\|AT(t)\| = O\left( \frac{1}{t^{1/\alpha+\varepsilon} } \right)  \qquad \text{as $t \downarrow 0$}
\end{equation}
for any $\varepsilon>0$ in \cite{Eberhardt1994}.
Moreover, in the Hilbert space setting,
the resolvent estimate \eqref{eq:resol_decay_poly_intro} 
implies
the semigroup estimate \eqref{eq:AT_growth_poly_intro},
as shown in \cite{Wakaiki2025_JFA}.

Whenever either of the estimates \eqref{eq:AT_growth_poly_intro}
or \eqref{eq:resol_decay_poly_intro} holds, 
$(T(t))_{t \geq 0}$ is of Gevrey class $\beta$ for all $\beta> 1/\alpha$; see \cite[Chapter~5]{Taylor1989}.
Gevrey class represents an intermediate level of regularity situated 
between holomorphic semigroups and immediately
differentiable semigroups. Therefore, 
to investigate the regularity
of solutions, resolvent estimates of the form \eqref{eq:resol_decay_poly_intro} 
were established for various types of
partial differential equations
and abstract evolution equations; see for instance \cite{Chen1990a,Graber2014,Hao2015,Badra2019,Kuang2021,Ammari2021,Avalos2021,
	Sozzo2022,Ammari2023,Rivera2024}.
Estimates for the growth rate of $\|AT(t)\|$ as $t\downarrow 0$
and the decay rate of $\|(i\eta - A)^{-1}\|$ as $|\eta| \to \infty$
were also used in examining the
differentiability of 
perturbed semigroups 
\cite{Pazy1968,Doytchinov1997,Arendt2006,Iley2007,Chen2020JEE}
and
delay semigroups \cite{Batty2004}.
A survey on this topic can be found in \cite{Batty2007}.

In this paper, we consider scales finer  than polynomial ones for
the rate of 
growth of $\|AT(t)\|$ as $t \downarrow 0$ and
decay of $\|(i\eta- A)^{-1}\|$
as $|\eta| \to \infty$.
Let $M \colon [0,\infty) \to (0,\infty)$ be a non-decreasing
continuous function such that $M(s) \to \infty$ and $M(s) = o(s)$ as $s \to \infty$.
Here $o(\cdot)$ denotes the little $o$ notation, whose
definition can be found  at the end
of this section.
We assume that $M$ is a function of positive increase.
The class of functions of positive increase was 
introduced in \cite{Rozendaal2019} for estimating the 
optimal rates of decay of $C_0$-semigroups; see 
Section~\ref{sec:positive_inc} for the definition of
functions of positive increase.
Functions of positive increase grow at least at a polynomial rate.
Hence, if the generator $A$ of a $C_0$-semigroup $(T(t))_{t \geq 0}$ with $\sigma(A) \cap i \mathbb{R}$ compact
satisfies
\begin{equation}
	\label{eq:resol_decay_intro}
	\|(i \eta - A)^{-1}\| \leq \frac{1}{M(|\eta|)}\quad 
	\text{for all $\eta \in \mathbb{R}$ with $|\eta|$ sufficiently large},
\end{equation}
then $(T(t))_{t \geq 0}$ is immediately differentiable by
\cite[Corollary~2.4.10]{Pazy1983}, which
is also stated in Corollary~\ref{coro:suff_id} of this paper.
We also note that the assumption that $M(s) = o(s)$ as $s \to \infty$ is not a strong restriction in the following sense: 
$M(s) = O(s)$ as $s \to \infty$
is necessary for the estimate \eqref{eq:resol_decay_intro}
to hold; see Remark~\ref{rem:M_O_s}.

First, we consider a Banach space 
semigroup $(T(t))_{t \geq 0}$.
For some sufficiently large $s_1\geq 0$, we can define $M_{\log}\colon
[s_1,\infty) \to (0,\infty)$ by
\[
M_{\log}(s) \coloneqq  \frac{M(s)}{\log(s/M(s))}.
\]
In Theorem~\ref{thm:non_exp}, we show that 
if $M_{\log}$ is non-decreasing, then the resolvent estimate 
\eqref{eq:resol_decay_intro} implies that
\begin{equation}
	\label{eq:semigroup_growth_intro}
	\|AT(t)\| = O \left(M_{\log}^{-1} \left(\frac{1}{ct} \right) \right)\qquad \text{as $t \downarrow 0$}
\end{equation}
for some constant $c \in (0,1)$, where 
$M_{\log}^{-1}$ is the left-inverse of $M_{\log}$.
Moreover,
we derive in
Theorem~\ref{thm:lower_bound} 
the following lower bound for an immediately differentiable $C_0$-semigroup $(T(t))_{t \geq 0}$ under certain assumptions:
\begin{equation}
	\label{eq:semigroup_growth_lower_intro}
	CN^{-1}\left(
	\frac{1}{ct}
	\right) \leq \sup_{t\leq \tau \leq 1}\|AT(\tau)\|
\end{equation}
for some constants $C,c>0$ and all sufficiently small $t>0$, where
the function $N$ is defined by
\begin{equation}
	\label{eq:N_def_intro}
	N(s) \coloneqq \frac{1}{\sup_{|\eta| \geq s} \|(i \eta- A)^{-1}\|}
\end{equation}
for sufficiently large $s \geq 0$.
If $M(s) = O(s^{\alpha})$ as $s \to \infty$ for some constant $\alpha \in (0,1)$, then
the semigroup estimate \eqref{eq:semigroup_growth_intro} can be written as
\[
\|AT(t)\| = O \left(\left(\frac{|\log t|}{t} \right)^{1/\alpha} \right)\qquad \text{as $t \downarrow 0$},
\]
which gives a better growth rate than the estimates 
\eqref{eq:CP_intro} and \eqref{eq:Eber_intro}.

There is a gap between the upper bound in \eqref{eq:semigroup_growth_intro} and the lower bound in \eqref{eq:semigroup_growth_lower_intro}.
For Hilbert space semigroups, the gap can be bridged
in Theorem~\ref{thm:Hilbert_case} by
refining the growth rate $M_{\log}^{-1} (1/(ct))$ in \eqref{eq:semigroup_growth_intro} 
to $M^{-1}(1/t)$.
This result generalizes the polynomial case where $M(s) = O(s^{\alpha})$
as $s \to \infty$ for some $\alpha \in (0,1]$, studied in \cite{Wakaiki2025_JFA}.
A sharper estimate can be derived in the case where 
$(T(t))_{t \geq 0}$ is a $C_0$-semigroup of normal operators on a Hilbert space
or a multiplication $C_0$-semigroup on an $L^p$-space with $1 \leq p < \infty$. To describe this estimate, we
let the function $N$ be as in \eqref{eq:N_def_intro}, and define 
\[
N_{\inf} (s) \coloneqq \inf_{\lambda >1} \frac{N(\lambda s)}{\log \lambda}
\]
for sufficiently large $s\geq 0$. We show in 
Theorem~\ref{thm:mul_lower_bound} that
if $(T(t))_{t \geq 0}$ is immediately differentiable and 
if $N(s) = o(s)$ as $s\to \infty$, then 
\[
(1-\varepsilon) N_{\inf}^{-1} \left( \frac{1}{t}\right) \leq \|AT(t)\| \leq (1+\varepsilon) N_{\inf}^{-1} \left( \frac{1}{t}\right)
\]
for all sufficiently small $t >0$
whenever $\varepsilon \in (0,1)$.

The approach adopted in this study is inspired by \cite{Chill2016,Rozendaal2019},
which investigated the relation between 
the decay of $\|T(t)(1-A)^{-1}\|$ as $t \to \infty$ and
the growth of $\|(i\eta - A)^{-1}\|$ as $|\eta| \to \infty$ for
a bounded $C_0$-semigroup $(T(t))_{t \geq 0}$.
These studies \cite{Chill2016,Rozendaal2019} built upon \cite{Borichev2010,Batty2016}, and 
the techniques developed in \cite{Borichev2010,Batty2016}
were used in \cite{Wakaiki2025_JFA} to prove that the resolvent estimate
\eqref{eq:resol_decay_poly_intro}
implies the semigroup estimate \eqref{eq:AT_growth_poly_intro}
for each $\alpha \in (0,1]$.
Therefore, 
it is a natural progression to employ the ideas from \cite{Chill2016,Rozendaal2019} to extend the results presented in \cite{Wakaiki2025_JFA}.

This paper is organized as follows:
In Section~\ref{sec:preliminaries}, 
we review the definitions and basic properties 
of left-inverses and functions of positive increase.
We also recall the characterization of immediately
differentiable $C_0$-semigroups.
In Section~\ref{sec:Banach}, we consider
Banach space semigroups and establish
an upper bound for the rate of growth of $\|AT(t)\|$
as $t \downarrow 0$ under assumptions
on the resolvent decay.
In Section~\ref{sec:lower_bound},
we present a lower bound for the rate of
the semigroup growth.
For Hilbert space semigroups, we improve the upper bound in
Section~\ref{sec:Hilbert}.
In Section~\ref{sec:nec_pi}, we show that 
the assumption of positive increase is necessary 
for this improved estimate to hold.
Finally, in Section~\ref{sec:multiplication},
a better quantified asymptotic result is given for $C_0$-semigroups of 
normal operators on Hilbert spaces 
and multiplication $C_0$-semigroups 
on $L^p$-spaces.

\paragraph{Notation}
We denote the set of positive integers by $\mathbb{N}$
and the set of non-negative integers by $\mathbb{N}_0$.
We write $\mathbb{R}_+ \coloneqq [0,\infty)$, 
$\overline{\mathbb{C}}_+ \coloneqq \{ z \in \mathbb{C}: \re z \geq 0 \}$, and 
$\mathbb{C}_- \coloneqq \{ z \in \mathbb{C}: \re z <0 \}$.
For $z \in \mathbb{C}$ and 
a subset $\Omega$ of  $\mathbb{C}$, we write
$\dist(z,\Omega) \coloneqq \inf\{ |z-\lambda|:\lambda \in \Omega  \}$.

Let $b \geq 0$ and 
$f,g\colon [b,\infty) \to (0,\infty)$.
We write 
\[
f(t) = O(g(t))	\qquad \text{as $t \to \infty$}
\]
if
there exist constants $C>0$ and $t_0 \geq b$ 
such that $f(t) \leq C g(t)$ for all $t \geq t_0$.
We write 
\[
f(t) = o(g(t))	\qquad \text{as $t \to \infty$}
\]	
if 
for all $\varepsilon>0$, there exists a constant $t_0 \geq b$
such that $f(t) \leq \varepsilon g(t)$ for all $t \geq t_0$.
We write 
\[
f(t) \asymp g(t)	\qquad \text{as $t \to \infty$}
\]	
if $f(t) = O(g(t))$ and $g(t) = O(f(t))$ as $t \to \infty$.
Analogous notation is used to describe
asymptotic behavior in other limits such as $t \downarrow 0$.

Throughout this paper, all Banach and Hilbert spaces are assumed to be complex.
Let $X$ be a Banach space. We denote by $\mathcal{L}(X)$
the algebra of bounded linear operators on $X$.
Let $A$ be a closed linear operator on $X$.
We write $D(A)$ for 
the domain of $A$.
We denote 
the spectrum of $A$ 
by $\sigma(A)$ and the resolvent set of $A$
by $\varrho(A)$.
For $z \in \varrho(A)$, we write $R(z,A) 
\coloneqq (z - A)^{-1}$.
We define the Fourier transform of $f \in L^1(\mathbb{R},X)$ by
\[
\mathcal{F}f(\eta) \coloneqq 
\int_{-\infty}^{\infty} e^{- i \eta t} f(t) dt,\quad \eta \in \mathbb{R}.
\]

\section{Preliminaries}
\label{sec:preliminaries}
In this section, we collect the fundamental definitions
and facts that will be used throughout this paper.
\subsection{Left-inverses}
\label{sec:left_inv}
First, we recall the definition and 
properties of left-inverses; see \cite{Embrechts2013} and \cite[Section~2.3]{Harjulehto2019}
for details of left-inverses.
Let $b \geq 0$ and let a non-decreasing function
$M \colon [b,\infty) \to \mathbb{R}_+$ satisfy 
$M(s) \to \infty$ as $s \to \infty$.
We denote 
the {\em left-inverse} of $M$
by $M^{-1}\colon [M(b),\infty) \to [b,\infty)$, which is
defined by
\[
M^{-1}(t) \coloneqq \inf\{ s \geq b: M(s) \geq t \}.
\]

The left-inverse $M^{-1}$ is non-decreasing and left-continuous.
Moreover, $M^{-1}(t) \to \infty$ as $t\to \infty$, and
\begin{equation}
	\label{eq:li_prop1}
	M^{-1}(M(s)) \leq s
\end{equation}
for all $s \geq b$.
If $M$ is continuous, then
\begin{equation}
	\label{eq:li_prop2}
	M(M^{-1}(t)) = t
\end{equation}
for all $t \geq M(b)$.
Except in Section~\ref{sec:estimate_AT_qm},
$M$ is considered to be continuous 
when the left-inverse $M^{-1}$ is used.
If $M$ is only left-continuous, then
\begin{equation}
	\label{eq:li_prop3}
	M(M^{-1}(t)) \leq t
\end{equation}
for all $t \geq  M(b)$.

The following lemma provides
a simple growth property of left-inverses.
\begin{lemma}
	\label{lem:left_inv_small_o}
	Let $b\geq 0$ and
	let $M\colon [b,\infty) \to \mathbb{R}_+$ be a 
	non-decreasing continuous function such that 
	$M(s) \to \infty$ as $s \to \infty$.
	If $M(s) = o(s)$ as $s \to \infty$,
	then $t = o(M^{-1}(t))$
	as $t \to \infty$.
\end{lemma}

\begin{proof}
	Let $\varepsilon >0$ be arbitrary. By assumption,
	there exists $s_1 \geq b$ such that
	\begin{equation}
		\label{eq:M_eps_bound}
		M(s) \leq \varepsilon s
	\end{equation}
	for all $s \geq s_1$.
	Since $M^{-1}(t) \to \infty$ as $t \to \infty$,
	there exists $t_0 \geq M(b)$ such that 
	\begin{equation}
		\label{eq:M_inv_1}
		M^{-1}(t) \geq 1
	\end{equation}
	for all $t \geq t_0$.
	Let $t \geq \max\{ M(s_1) + \varepsilon,\, t_0\}$.
	Since $M$ is continuous and satisfies 
	$M(s) \to \infty$ as $s \to \infty$, there exists
	$s_t \geq s_1$ such that 
	\begin{equation}
		\label{eq:t_def_st}
		t = M(s_t)+\varepsilon.
	\end{equation}
	By the definition of left-inverses, 
	\begin{equation}
		\label{eq:st_upper_bound}
		s_t \leq M^{-1}(M(s_t)+\varepsilon) = M^{-1}(t).
	\end{equation}
	The estimate \eqref{eq:M_eps_bound} leads to
	\[
	\frac{M(s_t)}{\varepsilon}  \leq s_t.
	\]
	This, together with \eqref{eq:t_def_st}
	and \eqref{eq:st_upper_bound}, yields
	\[
	\frac{t-\varepsilon }{\varepsilon} =
	\frac{M(s_t)}{\varepsilon}  \leq 
	M^{-1}(t).
	\]
	Using the inequality~\eqref{eq:M_inv_1}, we derive
	\[
	t \leq \varepsilon (M^{-1}(t) + 1) \leq 2 \varepsilon M^{-1}(t).
	\]
	Since $\varepsilon >0$ is arbitrary, it follows that 
	$t = o(M^{-1}(t))$ as $t \to \infty$.
\end{proof}

\subsection{Functions of positive increase}
\label{sec:positive_inc}
Next, we recall the definition of functions of
positive increase introduced in \cite{Rozendaal2019}.
Let $b \geq 0$.
We say that a measurable function
$M \colon [b,\infty) \to (0,\infty)$ 
{\em has  positive increase} if
there exist constants 
$\alpha >0$, $c_0 \in (0,1]$, and $s_0 \geq b$ such that 
\begin{equation}
	\label{eq:positive_inc_cond}
	\frac{M(\lambda s)}{M(s)} \geq c_0\lambda^\alpha
	\qquad \text{for all $\lambda \geq 1$ and $s \geq s_0$.} 
\end{equation}

If \eqref{eq:positive_inc_cond} holds, then
there exists a constant $c_1 >0$
such that
\begin{equation}
	\label{eq:M_poly_lower_bound}
	M(s) \geq  c_1 s^{\alpha}
\end{equation}
for all $s \geq s_0$. Therefore, functions of positive increase
grow at least at a polynomial rate.
For this class of functions, the following integral estimate
also holds.
\begin{lemma}
	\label{lem:integral_estimate}
	Let $b \geq 0$
	and let $M \colon [b,\infty) \to (0,\infty)$ be a measurable
	function. 
	Let constants
	$\alpha >0$, $c_0 \in (0,1]$, and $s_0 \geq b$ 
	satisfy \eqref{eq:positive_inc_cond}.
	Then for all $\gamma >2/\alpha$,
	\[
	\int^{\infty}_{s_0} \frac{s}{M(s)^{\gamma}} ds \leq 
	\frac{s_0^2}{(\alpha \gamma-2)c_0^\gamma M(s_0)^{\gamma}}.
	\]
\end{lemma}
\begin{proof}
	By \eqref{eq:positive_inc_cond},
	\[
	\frac{1}{M(s)} \leq 
	\left( \frac{s_0}{s} \right)^{\alpha} \frac{1}{c_0M(s_0)}
	\]
	for all $s \geq s_0$. If
	$\gamma > 2/\alpha$, then
	\begin{align*}
		\int^{\infty}_{s_0} \frac{s}{M(s)^{\gamma}} ds &\leq 
		\int^{\infty}_{s_0} \left( \left( \frac{s_0}{s} \right)^{\alpha} \frac{1}{c_0M(s_0)}
		\right)^{\gamma} sds \\
		&=
		\frac{s_0^{\alpha \gamma}}{c_0^\gamma M(s_0)^\gamma}
		\int^{\infty}_{s_0} \frac{1}{s^{\alpha \gamma-1}}ds \\
		&=\frac{s_0^2}{(\alpha \gamma -2)c_0^{\gamma}M(s_0)^{\gamma}}. 
	\end{align*}
\end{proof}

The next result was proved for 
right-inverses in \cite[Proposition~2.2]{Rozendaal2019}.
The same argument can be applied to 
the case of left-inverses.
\begin{lemma}
	\label{lem:RSS22}
	Let $b \geq 0$
	and let $M \colon [b,\infty) \to (0,\infty)$
	be a non-decreasing continuous function such that $M(s) \to \infty$ as $s \to \infty$.
	If $M$ has positive increase, then for all  $c>0$,
	\[
	M^{-1}(t) \asymp M^{-1}(ct)	\qquad \text{as $t \to \infty$}.
	\]
\end{lemma}

\subsection{Immediately differentiable $C_0$-semigroups}
Finally, we collect some results on immediately differentiable 
$C_0$-semigroups.
We will use the spectral property 
presented in the following 
characterization; see, e.g., \cite[Theorem~2.4.8]{Pazy1983}
and \cite[Corollary~II.4.15]{Engel2000} for the proof.
\begin{theorem}
	\label{thm:spectral_prop}
	Let $A$ be the generator of a $C_0$-semigroup 
	$(T(t))_{t \geq 0}$ on a Banach space.
	Let $\omega_0$
	be the exponential growth bound of $(T(t))_{t \geq 0}$,
	and let $\omega > \omega_0$. Then
	the following statements are equivalent:
	\begin{enumerate}
		\renewcommand{\labelenumi}{\rm{(\roman{enumi})}}
		\item $(T(t))_{t \geq 0}$ is immediately differentiable.
		\item For all $q >0$, there exist constants $p,C >0$
		such that 
		\[
		\Theta \coloneqq
		\left\{
		z \in \mathbb{C} : p e^{-q \re z}
		\leq |\im z|
		\right\} \subset \varrho(A)
		\] 
		and 
		\[
		\|R(z,A)\| \leq C |\im z|
		\]
		for all $z \in \Theta$ with $\re z \leq \omega$.
	\end{enumerate}
\end{theorem}
Theorem~\ref{thm:spectral_prop}
leads to a sufficient condition for
$C_0$-semigroups to be immediately differentiable.
The proof can be found, e.g., in \cite[Corollary~2.4.10]{Pazy1983}.
\begin{corollary}
	\label{coro:suff_id}
	Let $A$ be the generator of a $C_0$-semigroup 
	$(T(t))_{t \geq 0}$ on a Banach space. 	Let $\omega_0$
	be the exponential growth bound of $(T(t))_{t \geq 0}$.
	If there exists a constant $\omega > \omega_0$ such that 
	\[
	\limsup_{|\eta| \to \infty}\, \log |\eta|\, \|R(\omega+i\eta,A)\| = 0,
	\]
	then $(T(t))_{t \geq 0}$ is immediately differentiable.
\end{corollary}

Let $A$ be the generator of a $C_0$-semigroup
on a Banach space $X$, and let $z \in \varrho(A)$.
For  $t \geq 0$ and
$y \in D(A)$, we have
\begin{equation}
	\label{eq:semigroup_int_formula}
	e^{z t} y = T(t)y  + 
	\int^t_0 e^{z(t - \tau) }T(\tau) (z - A)yd\tau;
\end{equation}
see, e.g., \cite[Lemma~II.1.9]{Engel2000}.
Let $x \in X$ and set $y \coloneqq R(z,A)x$.
Differentiating both sides of \eqref{eq:semigroup_int_formula} 
with respect to $t$, we obtain
\begin{equation}
	\label{eq:resol_for_lower_bound}
	z  e^{z  t} R(z ,A)x = 
	AT(t) R(z ,A)x + T(t)x + z \int^t_0 e^{z  (t-\tau)}
	T(\tau) x d\tau
\end{equation}
for all $t >0$.
Using \eqref{eq:resol_for_lower_bound} as in the proof of
\cite[Theorem~2.1]{Crandall1969},
we obtain the following 
corollary as a further consequence of Theorem~\ref{thm:spectral_prop}.
\begin{corollary}
	Let $A$ be the generator of an immediately 
	differentiable $C_0$-semigroup 
	$(T(t))_{t \geq 0}$ on a Banach space $X$.
	Then there exists a constant $b\geq 0$ such that 
	$\sigma(A) \cap i \mathbb{R} \subset (-ib,ib)$.
	Moreover, for all $b \geq 0$ satisfying
	$\sigma(A) \cap i \mathbb{R} \subset (-ib,ib)$,
	\begin{equation}
		\label{eq:resol_bound_id}
		\sup_{|\eta| \geq b} \|R(i \eta,A)\| < \infty.
	\end{equation}
\end{corollary}
\begin{proof}
	The existence of a constant $b\geq 0$ satisfying
	$\sigma(A) \cap i \mathbb{R} \subset (-ib,ib)$
	follows immediately from 
	Theorem~\ref{thm:spectral_prop}. 
	Let $b \geq 0$ satisfy
	$\sigma(A) \cap i \mathbb{R} \subset (-ib,ib)$, and
	let $\eta \in \mathbb{R}$ satisfy $|\eta| \geq 
	\max\{b,\,1\}$.
	Define $C_T \coloneqq \sup_{0\leq t \leq 1} \|T(t)\|$.
	Substituting
	$z = i \eta$ and $t=1$ into
	\eqref{eq:resol_for_lower_bound},
	we obtain
	\[
	\|R(i\eta, A)\| \leq 
	\frac{1}{|\eta|}(
	\|AT(1)\|\, \|R(i\eta,A)\| + C_T
	) + C_T.
	\]
	This estimate gives
	$
	\|R(i\eta, A)\| \leq 
	4C_T
	$
	if $|\eta| \geq 2 \|AT(1)\|$.
	Since $\|R(\cdot ,A)\|$ is bounded on
	compact subsets of $\varrho(A)$, we conclude that
	the estimate \eqref{eq:resol_bound_id} holds.
\end{proof}

By the estimate \eqref{eq:resol_bound_id},
the function $M \colon [b,\infty) \to (0,\infty)$ given by
\[
M(s) \coloneqq \frac{1}{\sup_{|\eta| \geq s} \|R(i \eta,A)\|}
\]
is well defined.
This function will be used in Sections~\ref{sec:lower_bound}
and \ref{sec:multiplication}.

\section{Growth of Banach space semigroups}
\label{sec:Banach}
In this section, we consider 
a Banach space semigroup $(T(t))_{t \geq 0}$
with generator $A$ and derive an upper bound for 
the rate of growth of $\|AT(t)\|$ as $t \downarrow 0$
under certain assumptions on the rate of decay of 
$\|R(i \eta,A)\|$ as $|\eta| \to \infty$. 
Let $M \colon \mathbb{R}_+ \to (0,\infty)$
be a non-decreasing continuous function such that 
the following three conditions hold:
\begin{enumerate}
	[label=(C\arabic*), leftmargin=2.5em, labelsep=0.5em]
	\item $M$ has positive increase.
	\item $M(s) = o(s)$ as $s \to \infty$. In particular,
	there exists $\widetilde s_1 > 0$ such that 
	$M(s) < s$
	for all $s \geq \widetilde s_1$.
	\item There exists $s_1 \geq \widetilde s_1$ such that 
	the function $s \mapsto M(s)/\log(s/M(s))$ is
	non-decreasing on $[s_1,\infty)$.
\end{enumerate}
Under conditions (C1)--(C3), we define $M_{\log} \colon
[s_1,\infty) \to (0,\infty)$ by
\begin{equation}
	\label{eq:Mlog_def}
	M_{\log}(s) \coloneqq \frac{M(s)}{\log(s/M(s))}.
\end{equation}
Then $M_{\log}$ is non-decreasing and continuous. 
Since $M$ grows at least at a polynomial rate,
it follows that
$M_{\log}(s) \to \infty$
as $s \to \infty$. These properties of  $M_{\log}$  are used for
the left-inverse $M_{\log}^{-1}$.

The following result  provides an improvement on the 
semigroup estimates obtained in \cite[Theorem~2.3]{Crandall1969}
and \cite[Theorem~2]{Eberhardt1994}.
We prove it by using the argument in the proof of
\cite[Theorem~2.1]{Chill2016}.
\begin{theorem}
	\label{thm:non_exp}
	Let $A$ be the generator of a $C_0$-semigroup
	$(T(t))_{t \geq 0}$ on a Banach space 
	$X$ such that $\sigma(A) \cap i \mathbb{R} \subset (-i b, i b)$ for some $b \geq 0$.
	Let $M \colon \mathbb{R}_+ \to (0,\infty)$
	be a non-decreasing continuous  function such that conditions {\rm (C1)--(C3)} hold,
	and let $M_{\log} \colon [s_1,\infty) \to (0,\infty)$
	be defined by \eqref{eq:Mlog_def}.
	If
	\begin{equation}
		\label{eq:resol_bound_Banach}
		\|R( i \eta,A)\| \leq \frac{1}{M(|\eta|)}
	\end{equation}
	for all $\eta \in \mathbb{R}$ with $|\eta| \geq b$, 
	then $(T(t))_{t \geq 0}$ is 
	immediately differentiable, and there exists 
	a constant $c \in (0,1)$
	such that
	\begin{equation}
		\label{eq:AT_bound_Banach_Thm}
		\|AT(t)\| = O \left(M_{\log}^{-1} \left(\frac{1}{ct} \right) \right)
		\qquad \text{as $t \downarrow 0$}.
	\end{equation}
\end{theorem}

Note that the function $M_{\log}$ defined by \eqref{eq:Mlog_def}
differs from that used for 
semi-uniform stability in \cite{Batty2008} and \cite{Chill2016}.
In the case of semi-uniform stability,
$M_{\log}$ is non-decreasing by definition, but
this is not true in general for $M_{\log}$ defined by
\eqref{eq:Mlog_def}.
Therefore, we impose condition (C3).
In addition, condition (C1)
is not assumed in the studies on semi-uniform stability.
Condition (C1) will be used to apply
Lemma~\ref{lem:integral_estimate} and 
Corollary~\ref{coro:suff_id}.
It is also worth mentioning that a constant $\alpha$
satisfying \eqref{eq:positive_inc_cond} lies
in the interval $(0,1) $ under condition (C2), which can be immediately
seen from
\eqref{eq:M_poly_lower_bound}.

The remainder of this section is as follows.
In Section~\ref{sec:prop_Mlog}, 
we discuss the situations under which condition (C3) is satisfied.
In Section~\ref{sec:exp_case_Banach}, we prove Theorem~\ref{thm:non_exp} in the case 
when $(T(t))_{t \geq 0}$ is {\em exponentially stable}, i.e.,
there exist $C_T\geq 1$ and $\omega >0$ such that
$\|T(t)\| \leq C_T e^{-\omega t}$ for all $t \geq 0$.
In Section~\ref{sec:non_exp_case_Banach}, 
we give the proof of Theorem~\ref{thm:non_exp} 
in the general case.

\subsection{Sufficient conditions for $M_{\log}$
	to be non-decreasing}
\label{sec:prop_Mlog}
The following proposition presents
sufficient conditions for 
$M_{\log}$ to be non-decreasing.

\begin{proposition}
	\label{prop:M_log_si}
	Assume that 
	the function $M \colon \mathbb{R}_+ \to (0,\infty)$
	satisfies
	one of the following conditions:
	\begin{enumerate}
		\renewcommand{\labelenumi}{\rm{(\roman{enumi})}}
		\item
		There exist
		constants  $\gamma \in (0,1)$, $\delta\geq \gamma/(1-\gamma)$, and $s_1 	> 1$ such that 
		\[
		M(s) \leq s^{\gamma}\qquad \text{and} \qquad 
		\frac{M(\lambda s)}{M(s)} \geq \left(
		\frac{\log (\lambda s)}{\log s}
		\right)^{1+\delta}
		\]
		for all $\lambda \geq 1$ and $s \geq s_1$. 
		\item 
		There exist
		constants $\alpha,s_1 >0$ such that 
		\[
		M(s) \leq e^{-1/\alpha} s\qquad \text{and} \qquad 
		\frac{M(\lambda s)}{M(s)} \geq \lambda^\alpha
		\]
		for all $\lambda \geq 1$ and $s \geq s_1$. 
	\end{enumerate}
	Then 
	$M_{\log}$ defined by \eqref{eq:Mlog_def} 
	is strictly increasing on $[s_1,\infty)$.
\end{proposition}
\begin{proof}
	In both cases (i) and (ii), we have
	$M(s) < s$
	for all $s \geq s_1$.
	Define $\theta\colon[s_1,\infty) \to (0,\infty) $ by 
	\[
	\theta(s) \coloneqq \log \left( \frac{s}{M(s)}\right).
	\]
	Then $M_{\log} = M/\theta$ is strictly
	increasing on the interval $[s_1,\infty)$ if and only if
	\begin{align*}
		\frac{M(\sigma)}{M(s)}
		>
		\frac{\theta(\sigma)}{\theta(s)}
	\end{align*}
	for all $\sigma, s \in [s_1,\infty)$ satisfying
	$\sigma > s$.
	Since $M$ is strictly increasing under 
	both conditions (i) and (ii), it follows that 
	for all $\lambda > 1$ and $s \geq s_1$,
	\begin{align*}
		\theta(\lambda s) 
		< \theta(s) + \log \lambda.
	\end{align*}
	Therefore, it suffices to show that
	\begin{equation}
		\label{eq:M_lower_bound}
		\frac{M(\lambda s)}{M(s)} \geq
		1+
		\frac{\log \lambda}{\theta(s)}
	\end{equation}
	for all $\lambda > 1$ and $s \geq s_1$.
	
	First, we assume that condition (i) is satisfied.
	Then
	\[
	\frac{M(\lambda s)}{M(s)} \geq \left(
	1+
	\frac{\log \lambda }{\log s}
	\right)^{1+\delta} \geq 
	1 + (1+\delta) \frac{\log \lambda }{\log s}
	\]
	for all $\lambda > 1$ and $s \geq s_1$. 
	By $\delta \geq \gamma/(1-\gamma)$,
	we have
	\begin{equation}
		\label{eq:1_delta_bound}
		1+\delta \geq \frac{1}{1-\gamma}.
	\end{equation}
	Since $M(s) \leq s^{\gamma}$ for all $s \geq s_1$, 
	it follows that
	\begin{equation}
		\label{eq:theta_lower_bound}
		\theta(s) \geq (1-\gamma) \log s
	\end{equation}
	for all $s \geq s_1$.
	Combining \eqref{eq:1_delta_bound} and \eqref{eq:theta_lower_bound}, we obtain
	\[
	(1+\delta) \frac{\log \lambda }{\log s} 
	\geq \frac{\log \lambda}{\theta(s)}
	\]
	for all $\lambda > 1$ and $s \geq s_1$.
	Thus, the desired inequality \eqref{eq:M_lower_bound}
	holds for all $\lambda > 1$ and $s \geq s_1$.
	
	Next, we assume that condition (ii) is satisfied.
	Then the desired inequality
	\eqref{eq:M_lower_bound} holds if
	\begin{align}
		\label{eq:theta_suff_estimate}
		\theta(s)
		\geq  \frac{\log \lambda}{\lambda^{\alpha}-1}.
	\end{align}
	For all $\lambda > 1$,
	\[
	\frac{\log \lambda}{  \lambda^{\alpha} - 1}
	<
	\lim_{\lambda \downarrow 1} \frac{\log \lambda}{  \lambda^{\alpha} - 1}= \frac{1}{\alpha }.
	\]
	Since $M(s) \leq e^{-1/\alpha} s
	$
	is equivalent to $\theta(s) \geq 1/\alpha$,
	we obtain \eqref{eq:theta_suff_estimate}
	for all $\lambda > 1$ and $s \geq s_1$.
\end{proof}

A similar argument
shows that $M_{\log}$ has positive increase if 
condition (ii) of Proposition~\ref{prop:M_log_si} holds.
Therefore,
by Lemma~\ref{lem:RSS22},
we can omit the constant $c$ in the estimate \eqref{eq:AT_bound_Banach_Thm} under condition (ii).
A typical example satisfying this condition is a polynomial. 
We now briefly investigate the case when $M$ is a polynomial.
\begin{example}
	Let $0< \alpha < 1$ and $C>0$.
	Define $M\colon \mathbb{R}_+ \to (0,\infty)$ by
	\[
	M(s) \coloneqq 
	\begin{cases}
		C, & 0 \leq s < 1, \\
		Cs^{\alpha}, & s \geq 1,
	\end{cases}
	\]
	which satisfies 
	condition (ii) of Proposition~\ref{prop:M_log_si}. 
	A routine calculation shows that 
	\[
	M_{\log}^{-1} (t)
	\asymp
	(t \log t)^{1/\alpha}\qquad \text{as $t \to \infty$}.
	\]
	This implies that 
	\[
	M_{\log}^{-1} 
	\left( \frac{1}{ct} \right) \asymp 
	\left(
	\frac{|\log t|}{t}
	\right)^{1/\alpha}\qquad \text{as $t \downarrow 0$}
	\]
	for all $c >0$.
\end{example}

\subsection{Case when $(T(t))_{t \geq 0}$
	is exponentially stable}
\label{sec:exp_case_Banach}
Here we prove Theorem~\ref{thm:non_exp}
under the assumption that
$(T(t))_{t \geq 0}$
is exponentially stable.
Throughout Section~\ref{sec:exp_case_Banach},
for $R>0$, define 
a continuously differentiable function
$\phi_R\colon \mathbb{R} \to \mathbb{R}$ by
\begin{equation}
	\label{eq:phi_R_def}
	\phi_R(t) \coloneqq R\phi(Rt),\quad 
	\text{where~}
	\phi(t) \coloneqq
	\begin{dcases}
		\frac{\cos(t) - \cos(2t)}{\pi t^2}, &t\neq 0, \\
		\frac{3}{2\pi}, &  t=0.
	\end{dcases}
\end{equation}
Then 
$\psi_R \coloneqq \mathcal{F}\phi_R$ is given by
\[
\psi_R(\eta) =
\begin{dcases}
	1,& |\eta| \leq R, \\
	2-\frac{|\eta|}{R}, & R \leq |\eta| \leq 2R, \\
	0 ,& |\eta| \geq 2R.
\end{dcases}
\]

Let $A$ be the generator of a $C_0$-semigroup  $(T(t))_{t \geq 0}$ on a 
Banach space $X$.
For $x \in D(A)$, we define $f_x \colon \mathbb{R} \to X$ by
\begin{equation}
	\label{eq:fx_def}
	f_x(t) \coloneqq
	\begin{cases}
		AT(t)x ,& t >0, \\
		0, & t \leq 0.
	\end{cases}
\end{equation}
We estimate $\|f_x(t)\|$ by using the decomposition
\[
f_x = \phi_R*  f_x + (f_x - \phi_R*f_x).
\]
The following lemma gives an estimate for $\phi_R*f_x $.
\begin{lemma}
	\label{lem:phiR_f_bound}
	Let $A$ be the generator of a bounded $C_0$-semigroup  $(T(t))_{t \geq 0}$ 
	on a 
	Banach space $X$, and let $C_T \coloneqq \sup_{t \geq 0} \|T(t)\|$. 
	Let $\phi$, $\phi_R$, and $f_x$ be as above.
	Then
	\[
	\|(\phi_R * f_x)(t)\| \leq 
	R(C_T +1) \|\phi'\|_{L^1(\mathbb{R})} \|x\| 
	\]
	for all $t \in \mathbb{R}$, $R>0$, and $x \in D(A)$.
\end{lemma}
\begin{proof}
	Let $x \in D(A)$ and
	define $F_x\colon \mathbb{R} \to X$ by
	\[
	F_x(t) \coloneqq \int^t_0 f_x(\tau) d\tau.
	\]
	Then
	\begin{equation}
		\label{eq:Fx_bound}
		\|F_x(t)\| \leq \|T(t)x - x\| \leq (C_T +1) \|x\|
	\end{equation}
	for all $t \geq 0$.
	For all $t \in \mathbb{R}$ and $R>0$,
	integration by parts gives
	\begin{align*}
		(\phi_R*f_x)(t)
		=
		R \int_0^{\infty} \phi(R(t-\tau)) f_x(\tau) d\tau 
		=R \int_0^{\infty} \phi'(Rt-\tau) F_x\left(
		\frac{\tau}{R}
		\right) d\tau.
	\end{align*}
	By \eqref{eq:Fx_bound}, we obtain
	\[
	\|(\phi_R * f_x)(t)\| \leq 
	R(C_T +1) \|\phi'\|_{L^1(\mathbb{R})} \|x\| 
	\]
	for all $t \in \mathbb{R}$ and $R>0$.
\end{proof}

\begin{proof}{Proof of Theorem~\ref{thm:non_exp}
		in the case when $(T(t))_{t \geq 0}$
		is exponentially stable} We assume that the $C_0$-semigroup $(T(t))_{t \geq 0}$
	is exponentially stable and that $b = 0$.
	Then the resolvent estimate \eqref{eq:resol_bound_Banach} holds for all $\eta \in \mathbb{R}$.
	
	The proof is divided into three steps.
	In Step~1, we decompose $f_x - \phi_R*f_x$ into three components and 
	obtain an upper bound for each.
	In Step~2, we combine these upper bounds in terms of $M_{\log}^{-1}$, which
	leads to an overall estimate for
	$\|f_x - \phi_R*f_x\|$. In Step~3, using
	this result and Lemma~\ref{lem:phiR_f_bound}, we derive an estimate for
	the rate of growth of $\|AT(t)\|$ as $t \downarrow 0$.
	
	Step~1.
	Let $x \in D(A^3)$.
	Then,
	for all $t >0$ and $n \in \mathbb{N}$,
	\[
	f_x(t) = \frac{n!}{2\pi t^{n}}
	\int^{\infty}_{-\infty} e^{i\eta t} AR(i\eta,A)^{n+1}x d\eta;
	\]
	see, e.g., \cite[Corollary~III.5.16]{Engel2000}.
	On the other hand, since
	\[
	(\phi_R * f_x)(t) = 
	\frac{1}{2\pi}\int^{2R}_{-2R} e^{i \eta t} \psi_R(\eta)AR(i \eta, A) x d\eta
	\]
	for all $t\in \mathbb{R}$, integrating by parts yields
	\begin{align*}
		&(\phi_R * f_x)(t) \\
		&~=
		\frac{n!}{2\pi t^n}\int^{2R}_{-2R} e^{i \eta t} \psi_R(\eta)AR(i \eta, A)^{n+1} x d\eta \\
		&\quad + 
		\frac{n!}{2\pi i Rt^n}
		\left(
		\int_{R}^{2R} 
		e^{i \eta t}AR(i \eta, A)^{n} x
		d\eta -
		\int_{-2R}^{-R} e^{i \eta t}AR(i \eta, A)^{n} x
		d\eta 
		\right)\\
		&\quad - 
		\frac{1}{2\pi Rt^2}
		\sum_{k=0}^{n-2} \frac{(k+1)!}{t^{k}}
		\Big(
		e^{2iRt} AR(2iR , A)^{k+1} x +
		e^{-2iRt} AR(-2i R, A)^{k+1} x
		\\
		&\hspace{115pt}
		-
		e^{iRt} AR(iR , A)^{k+1} x 
		-
		e^{-iRt} AR(-i R, A)^{k+1} x
		\Big)
	\end{align*}
	for all $t,R >0$ and $n \in \mathbb{N}$.
	Hence,
	\[
	f_x(t) - (\phi_R*f_x)(t) =J_1 + J_2 + J_3
	\]
	for all $t,R >0$ and $n \in \mathbb{N}$, where
	\begin{align*}
		J_1 &\coloneqq \frac{n!}{2\pi t^{n}}
		\int^{\infty}_{-\infty} e^{i\eta t} (1 - \psi_R(\eta))
		AR(i\eta,A)^{n+1}x d\eta , \\
		J_2 &\coloneqq -
		\frac{n!}{2\pi i Rt^n}
		\left(
		\int_{R}^{2R} 
		e^{i \eta t}AR(i \eta, A)^{n} x
		d\eta -
		\int_{-2R}^{-R} e^{i \eta t}AR(i \eta, A)^{n} x
		d\eta 
		\right) , \\
		J_3 &\coloneqq
		\frac{1}{2\pi R t^2}
		\sum_{k=0}^{n-2} \frac{(k+1)!}{t^{k}}
		\Big(
		e^{2iRt} AR(2iR , A)^{k+1} x +
		e^{-2iRt} AR(-2i R, A)^{k+1} x
		\\
		&\hspace{105pt}
		-
		e^{iRt} AR(iR , A)^{k+1} x 
		-
		e^{-iRt} AR(-i R, A)^{k+1} x
		\Big).
	\end{align*}

	For all $\eta \in \mathbb{R}$,
	\[
	AR(i\eta ,A) = i\eta R(i\eta ,A) - 1.
	\]
	By assumption \eqref{eq:resol_bound_Banach},
	\[
	\|AR(i\eta ,A)\| \leq \frac{|\eta|}{M(|\eta|)} + 1
	\]
	for all $\eta \in \mathbb{R}$.
	Since
	$M(s) < s$ for all $s \geq s_1$, we have
	\[
	\|AR(i\eta ,A)\| \leq \frac{2|\eta|}{M(|\eta|)}
	\]
	for all $\eta \in \mathbb{R}$ with $|\eta| \geq 
	s_1$.
	This gives
	\begin{equation}
		\label{eq:AR_n_bound_Banach}
		\|AR(i\eta ,A)^n\| \leq \frac{2|\eta|}{M(|\eta|)^n}
	\end{equation}
	for all $n \in \mathbb{N}$ and 
	$\eta \in \mathbb{R}$ with $|\eta| \geq 
	s_1$.
	
	Let $\alpha >0$, $c_0 \in (0,1]$, and $s_0 \geq 0$ 
	satisfy the inequality \eqref{eq:positive_inc_cond}.
	Define $n_{\alpha} \coloneqq 
	\min\{n \in \mathbb{N}:n > 2/\alpha\}$ and $\beta \coloneqq
	\alpha n _{\alpha} - 2 >0$.
	Let $R \geq \max\{s_0,\, s_1\} \eqqcolon s_2$ 
	and let $n \in \mathbb{N}$ with $n \geq n_{\alpha}$. 
	By
	the estimate~\eqref{eq:AR_n_bound_Banach} and 
	Lemma~\ref{lem:integral_estimate}, we have
	\begin{align}
		\|J_1\|&\leq 
		\frac{2n! \|x\|}{\pi t^{n}}
		\int^{\infty}_{R}  \frac{s}{M(s)^{n+1}}ds \leq \frac{2\|x\|}{c_0 \beta  \pi  }
		\frac{n!R^2}{c_0^{n}t^nM(R)^{n+1}} 
		\label{eq:J1_bound}
	\end{align}
	for all $t > 0$.
	Similarly,
	\begin{align*}
		\|J_2\|&\leq 
		\frac{2n! \|x\|}{\pi R t^n}
		\int_{R}^{\infty} \frac{s}{M(s)^n}
		ds \leq 
		\frac{2\|x\|}{\beta \pi}
		\frac{n!R^2}{c_0^nt^nRM(R)^n}
	\end{align*}
	for all $t > 0$.
	Since 
	$M(s) < s$ for all $s \geq s_1$,
	it follows that 
	\begin{equation}
		\label{eq:J2_bound}
		\|J_2\|\leq 
		\frac{2\|x\|}{\beta \pi}
		\frac{n!R^2}{c_0^nt^nM(R)^{n+1}}
	\end{equation}
	for all $t > 0$.
	The estimate \eqref{eq:AR_n_bound_Banach} also yields
	\begin{align}
		\|J_3\| & \leq 
		\frac{2n \|x\|}{ \pi R t^{2}}
		\sum_{k=0}^{n-2}
		\frac{k!}{t^k}
		\left(
		\frac{2R}{M(2R)^{k+1}} + \frac{R}{M(R)^{k+1}} 
		\right) \notag \\
		&\leq
		\frac{6n \|x\|}{\pi t^{2}M(R)}
		\sum_{k=0}^{n-2}
		\frac{k!}{t^kM(R)^{k}}
		\label{eq:J3_bound}
	\end{align}
	for all $t > 0$.
	
	Step~2.
	Let $0 < \gamma  < 1$.
	By the estimates \eqref{eq:J1_bound}--\eqref{eq:J3_bound},
	here 
	we show that 
	there exist $C_1,t_1>0$ such that 
	for all $t \in (0, t_1]$ and $x \in D(A^3)$,
	\begin{equation}
		\label{eq:fx_phi_fx}
		\|f_x(t) - (\phi_R*f_x)(t)\| \leq C_1 \|x\| M_{\log}^{-1} \left(
		\frac{1}{\gamma c_0t}
		\right).
	\end{equation}
	To this end, we set
	\[
	t_2 \coloneqq \frac{1}{\gamma c_0M_{\log}(s_2)}.
	\]
	For $t \in (0,t_2)$, define 
	\begin{equation}
		\label{eq:R_def_Banach}
		R \coloneqq M_{\log}^{-1} \left(
		\frac{1}{\gamma c_0t}
		\right).
	\end{equation}
	Then $R \geq s_2$.
	The property \eqref{eq:li_prop2}
	of left-inverses yields
	$M_{\log}(R) = 1/(\gamma c_0t)$. Hence,
	\begin{equation}
		\label{eq:qctM}
		\gamma c_0t M(R) =
		\frac{M(R)}{M_{\log}(R)} = 
		\theta(R)
	\end{equation}
	for all $t \in (0, t_2)$,
	where $\theta(R) \coloneqq \log(R/M(R))$.
	Since $\theta(s) \to \infty$ and 
	$M_{\log}^{-1}(s) \to \infty$ as $s \to \infty$,
	there exists $t_0 \in (0, t_2)$ such that 
	$\gamma c_0tM(R) \geq n_{\alpha}$ for all $t \in (0, t_0]$.
	
	Let $0 < t \leq t_0$ and
	define 
	\begin{equation}
		\label{eq:n_def_Banach}
		n \coloneqq \max \{ m \in \mathbb{N}: m \leq \gamma c_0tM(R) \} \geq n_{\alpha},
	\end{equation}
	where $R$ is as in \eqref{eq:R_def_Banach}.
	By Stirling's formula,
	there exists $C_0>0$ such that 
	\[
	m! \leq C_0\left ( \frac{m}{e\gamma} \right)^m
	\]
	for all $m \in \mathbb{N}$.
	Then we have the following 
	estimate on the right-hand terms of \eqref{eq:J1_bound}
	and \eqref{eq:J2_bound}:
	\begin{equation}
		\label{eq:nR_tMR_bound}
		\frac{n!R^2}{c_0^{n}t^{n}M(R)^{n+1}}
		\leq \frac{C_0R^2}{M(R)}
		\left(
		\frac{n}{e\gamma c_0 tM(R)}
		\right)^{n} 
		\leq \frac{eC_0R^2}{M(R)}e^{-\gamma c_0t M(R)}.
	\end{equation}
	From \eqref{eq:qctM}, it follows that
	\begin{equation}
		\label{eq:expTMR}
		e^{-\gamma c_0t M(R)} = e^{-\theta(R)} =
		\frac{M(R)}{R}.
	\end{equation}
	By \eqref{eq:nR_tMR_bound} and \eqref{eq:expTMR},
	\begin{equation}
		\label{eq:J1_right}
		\frac{n!R^2}{c_0^{n}t^{n}M(R)^{n+1}}
		\leq eC_0R
		= eC_0  M_{\log}^{-1} \left(
		\frac{1}{\gamma c_0t}
		\right).
	\end{equation}
	Using $n \leq \gamma c_0tM(R)$, we also have 
	the following estimate for
	the right-hand term of \eqref{eq:J3_bound}:
	\begin{align}
		\label{eq:tM_sum_bound}
		\frac{n}{t^{2}M(R)}
		\sum_{k=0}^{n-2}
		\frac{k!}{t^kM(R)^{k}}
		&\leq
		\frac{\gamma c_0}{t}
		\sum_{k=0}^{n-2}
		\left(\frac{k}{tM(R)} \right)^k \notag \\
		&\leq 
		\frac{\gamma c_0}{t}
		\sum_{k=0}^{n-2}
		(\gamma c_0)^k 
		\leq 
		\frac{\gamma c_0}{(1-\gamma c_0)t}.
	\end{align}
	Since $M(s) = o(s)$ and $M_{\log}(s) = O(M(s))$
	as $s \to \infty$, Lemma~\ref{lem:left_inv_small_o} shows that there exists $t_1 \in (0,t_0]$
	such that 
	\begin{equation}
		\label{eq:1_tau_Mlog_bound}
		\frac{1}{\tau} \leq M_{\log}^{-1} \left(
		\frac{1}{\gamma c_0\tau}
		\right)
	\end{equation}
	for all $\tau \in (0, t_1]$.
	If $t \leq t_1$, then the estimates \eqref{eq:tM_sum_bound} and 
	\eqref{eq:1_tau_Mlog_bound} give
	\begin{equation}
		\label{eq:J3_right}
		\frac{n}{t^{2}M(R)}
		\sum_{k=0}^{n-2}
		\frac{k!}{t^kM(R)^{k}} \leq \frac{\gamma c_0}{1-\gamma c_0}M_{\log}^{-1} \left(
		\frac{1}{\gamma c_0t}
		\right).
	\end{equation}
	Applying \eqref{eq:J1_right} and \eqref{eq:J3_right} to \eqref{eq:J1_bound}--\eqref{eq:J3_bound}, we conclude that
	there exists $C_1>0$ such that the desired estimate~\eqref{eq:fx_phi_fx} holds
	for all $t \in (0, t_1]$ and $x \in D(A^3)$.
	
	Step~3.
	By  Lemma~\ref{lem:phiR_f_bound},
	there exists $C_2>0$ such that 
	for all $t \in (0, t_1]$ and $x \in D(A)$,
	\[
	\|(\phi_R*f_x)(t)\| \leq C_2 \|x\|M_{\log}^{-1} \left(
	\frac{1}{\gamma c_0t}
	\right),
	\]
	where $R$ is as in \eqref{eq:R_def_Banach}.
	This and \eqref{eq:fx_phi_fx} establish that
	\[
	\|AT(t)x\| = \|f_x(t)\| \leq (C_1+C_2)\|x\|M_{\log}^{-1} \left(
	\frac{1}{\gamma c_0t}
	\right)
	\]
	for all $t \in (0, t_1]$ and $x \in D(A^3)$.
	By \eqref{eq:M_poly_lower_bound}
	and \eqref{eq:resol_bound_Banach}, we have $\|R( i \eta,A)\| = O(|\eta|^{-\alpha})$
	as $|\eta| \to \infty$. Hence,
	Corollary~\ref{coro:suff_id} shows that
	$(T(t))_{t \geq 0}$ is immediately differentiable.
	It follows that 
	$AT(t) \in \mathcal{L}(X)$ for all $t >0$.
	By the density of $D(A^3)$ in $X$, we obtain
	\[
	\|AT(t)\| \leq (C_1+C_2)M_{\log}^{-1} \left(
	\frac{1}{\gamma c_0t}
	\right)
	\]
	for all $t \in (0, t_1]$.
\end{proof}

\subsection{Case when $(T(t))_{t \geq 0}$
	may not be exponentially stable}
\label{sec:non_exp_case_Banach}
Finally, we prove Theorem~\ref{thm:non_exp} 
in the case when $(T(t))_{t \geq 0}$
may not be exponentially stable.

\begin{proof}{Proof of Theorem~\ref{thm:non_exp}}
	Let $C_T \geq 1$ and $\omega >0$ satisfy
	$\|T(t)\| \leq C_T  e^{\omega t}$ for all $t \geq 0$.
	Define $B \coloneqq A - 2\omega$.
	Then $B$ generates an 
	exponentially stable $C_0$-semigroup $(S(t))_{t\geq 0}$.
	The resolvent equation gives
	\begin{align*}
		\|R(i\eta, B)\| 
		&\leq 
		\|R(i \eta,A)\| + 2\omega \|R(i\eta,A)\|\, \|R(2\omega+i\eta,A)\|.
	\end{align*}
	By the Hille-Yosida theorem,
	\[
	\|R(2\omega+i\eta,A)\| \leq \frac{C_T }{\omega}
	\]
	for all $\eta \in \mathbb{R}$.
	Hence, by assumption~\eqref{eq:resol_bound_Banach},
	\[
	\|R(i\eta, B)\|  \leq  (1+2C_T ) \|R(i\eta,A)\| \leq \frac{1+2C_T }{ M(|\eta|)}
	\]
	for all $\eta \in \mathbb{R}$ with $|\eta| \geq b$.
	In addition, we have 
	\[
	\sup_{-b < \eta < b} \|R(i\eta,B)\| < \infty.
	\]
	By these estimates,
	there exists $\delta >1$ such that 
	\[
	\|R(i\eta,B)\| \leq \frac{\delta}{ M(|\eta|)} 
	\]
	for all $\eta \in \mathbb{R}$. 
	
	Define
	$N \colon \mathbb{R}_+ \to (0,\infty)$ by $N(s) \coloneqq M(s)/\delta$ and 
	$N_{\log} \colon [s_1,\infty) \to (0,\infty)$ by
	\[
	N_{\log}(s ) \coloneqq \frac{N(s)}{\log(s/N(s))}.
	\]
	Since $M_{\log}$ is non-decreasing on $[s_1,\infty)$, 
	it follows that
	$N_{\log}$ is also non-decreasing on $[s_1,\infty)$.
	Indeed,
	we have
	\[
	\frac{1}{\delta N_{\log}(s) } =
	\frac{\log(s/M(s))}{M(s)} + \frac{\log \delta }{M(s)} =
	\frac{1}{M_{\log}(s) } + \frac{\log \delta }{M(s)}.
	\]
	Therefore, $1/N_{\log}$ is non-increasing
	on 
	$[s_1,\infty)$.
	
	By the argument in 
	Section~\ref{sec:exp_case_Banach},
	$(S(t))_{t \geq 0}$ is immediately differentiable, and 
	there exists $c \in (0,1)$
	such that 
	\[
	\|BS(t)\| = O \left(N_{\log}^{-1} \left( \frac{1}{ct} \right) \right)
	\qquad \text{as $t\downarrow 0$}.
	\]
	Then $(T(t))_{t \geq 0}$ is also immediately differentiable, and
	\begin{equation}
		\label{eq:AT_bound_Banach}
		\|AT(t)\| \leq 
		2\omega \|T(t)\| + e^{2\omega t} \|BS(t)\| = 
		O \left(N_{\log}^{-1} \left( \frac{1}{ct} \right) \right)
		\qquad \text{as $t\downarrow 0$}.
	\end{equation}
	It remains to replace $N_{\log}^{-1}$ by $M_{\log}^{-1}$
	in the estimate \eqref{eq:AT_bound_Banach}.
	Since $N(s) = o(s)$ as $s \to \infty$,
	there exists $s_{\delta} \geq s_1$ such that
	$\log(s/N(s)) \geq 2 \log \delta$
	for all $s \geq s_{\delta}$.
	Hence,
	\begin{align*}
		M_{\log}(s) 
		= \frac{\delta N(s)}{\log(s/N(s)) - \log \delta}
		\leq \frac{2\delta N(s)}{\log(s/N(s))}=
		2\delta N_{\log}(s)
	\end{align*}
	for all $s \geq s_{\delta}$.
	For all $t >N_{\log}(s_{\delta})$, we also have
	\begin{align*}
		N_{\log}^{-1}(t) &=
		\inf\{
		s \geq s_{\delta}:
		N_{\log }(s) \geq t
		\} \\
		&\leq 
		\inf\{
		s \geq s_{\delta}:
		M_{\log }(s) \geq 2\delta t
		\} =
		M_{\log}^{-1}(2\delta t).
	\end{align*}
	From this and the estimate \eqref{eq:AT_bound_Banach},
	we conclude that 
	\[
	\|AT(t)\| = O\left(M_{\log}^{-1}\left(\frac{1}{(c/2\delta ) t}
	\right) \right) \qquad \text{as $t \downarrow 0$}. 
	\]
\end{proof}

\section{Lower bound for growth rates}
\label{sec:lower_bound}
In this section, we show how certain 
semigroup estimates can be transferred to
resolvent estimates. 
We extend the argument in the proof of
\cite[Theorem~2.1]{Crandall1969} to a more general setting.
By this result, we also derive
a lower bound for the growth rate of 
$C_0$-semigroups.
\begin{theorem}
	\label{thm:lower_bound}
	Let $A$ be the generator of an
	immediately differentiable $C_0$-semigroup
	$(T(t))_{t \geq 0}$ on a Banach space $X$.
	Then the following statements hold:
	\begin{enumerate}
		\renewcommand{\labelenumi}{\rm{\alph{enumi})}}
		\item Let $K\colon \mathbb{R}_+\to \mathbb{R}_+$
		be a non-decreasing continuous
		function such that $t = O(K(t))$ as $t\to \infty$.
		If 
		\begin{equation}
			\label{eq:AT_estimate_K}
			\|AT(t)\| = O\left(K \left( \frac{1}{t} \right)\right)\qquad \text{as $t \downarrow 0$},
		\end{equation}
		then
		there exists a constant $c \in (0,1)$ such that
		\begin{equation}
			\label{eq:resol_estimate_K_inv}
			\|R(i\eta,A)\| = O 
			\left(
			\frac{1}{K^{-1}(c|\eta|)}
			\right)\qquad \text{as $|\eta| \to \infty$}.
		\end{equation}
		\item 
		Let $b \geq 0$ satisfy $\sigma(A) \cap i \mathbb{R} \subset (-ib,ib)$.
		Define $M\colon [b,\infty) \to (0,\infty) $ 
		and 
		$K\colon [1,\infty) \to \mathbb{R}_+ $ 
		by
		\begin{align*}
			M(s) &\coloneqq \frac{1}{\sup_{|\eta| \geq s} \|R(i\eta,A)\|},\\
			K(t) &\coloneqq \sup_{1\leq \tau \leq t }
			\left\|AT \left( \frac{1}{\tau} \right) \right\|.
		\end{align*}
		Assume that $M(s) \to \infty$ as $s \to \infty$ and that
		$
		t = O
		\left( K(t)
		\right)
		$
		as $t \to \infty$.
		Then 
		there exists a constant $c >0$ such that 
		\begin{equation}
			\label{eq:M_inv_K}
			M^{-1}\left(
			\frac{1}{ct}
			\right) = O\left(K \left( \frac{1}{t} \right )\right)\qquad \text{as $t \downarrow 0$}.
		\end{equation}
	\end{enumerate}
\end{theorem}

\begin{proof}
	a)
	Let $C_T \coloneqq \sup_{0\leq t \leq 1} \|T(t)\|$ and
	let $b > 0$ satisfy $\sigma(A) \cap i \mathbb{R} \subset (-ib,ib)$.
	Using \eqref{eq:resol_for_lower_bound} with $z = i \eta$, we obtain
	\[
	\|R(i\eta,A)\|
	\leq 
	\frac{1}{|\eta|}
	\left(
	\|AT(t)\|\, \|R(i\eta,A)\| + C_T
	\right) + C_T t
	\]
	for all $t \in (0,1]$  and $\eta \in \mathbb{R}$ with
	$|\eta| \geq  b$.
	By assumption~\eqref{eq:AT_estimate_K},
	there exist $C_0>0$ and $t_0 \in (0,1]$ such that 
	$\|AT(t)\| \leq C_0 K(1/t)$ for all $t \in (0,t_0]$.
	This implies that 
	\begin{align*}
		\left(
		1 - \frac{C_0K(1/t)}{|\eta|}
		\right)
		\|R(i\eta,A)\| 
		&\leq 
		\left(
		1 - \frac{\|AT(t)\|}{|\eta|}
		\right)
		\|R(i\eta,A)\|  \\
		&\leq C_T
		\left(
		\frac{1}{|\eta|} + t
		\right)
	\end{align*}
	for all $t \in (0,t_0]$  and $\eta \in \mathbb{R}$ with
	$|\eta| \geq  b$.
	
	Let $c \in (0,1/C_0)$ and $\eta \in \mathbb{R}$ 
	with $|\eta| \geq \max\{ b,\, (K(1/t_0) + 1)/c \}$. 
	Set
	\[
	t_{\eta} \coloneqq \frac{1}{K^{-1}(c|\eta|)}.
	\] 
	Then $t_{\eta} \leq t_0$, and the property \eqref{eq:li_prop2} 
	of left-inverses yields
	\[
	K\left( \frac{1}{t_{\eta}} \right) = K (K^{-1}(c|\eta|)) = c|\eta|.
	\]
	Hence,
	\[
	\left(1- cC_0 \right)\|R(i\eta,A)\| \leq C_T
	\left(
	\frac{1}{|\eta|} + \frac{1}{K^{-1}(c|\eta|)}
	\right).
	\]
	Since $t = O(K(t))$ as $t \to \infty$,
	there exist 
	$C>0$ and $t_1 \geq 0$ such that $t \leq CK(t)$ 
	for all $t \geq t_1$.
	Therefore,
	if $|\eta| \geq t_1/(cC) \eqqcolon s_1$, then
	\begin{align*}
		K^{-1}(c|\eta|)
		\leq \inf \{ t \geq t_1 : K(t) \geq c|\eta| \} 
		\leq 
		cC|\eta|. 
	\end{align*}
	This yields
	\[
	\frac{1}{|\eta|} \leq \frac{cC}{K^{-1}(c|\eta|)}
	\]
	whenever $|\eta| \geq s_1$.
	Thus,
	the desired estimate \eqref{eq:resol_estimate_K_inv}
	holds.
	
	b)
	Define a piecewise linear function
	$L\colon [1,\infty) \to \mathbb{R}_+$ by
	\[
	L(t) \coloneqq \frac{K(2^{n+1})}{2^n}(2^{n+1}-t) +
	\frac{K(2^{n+2})}{2^n}(t-2^n)
	\]
	for 
	$2^n\leq t < 2^{n+1}$ and $n \in \mathbb{N}_0$.
	Then $L$ is non-decreasing and satisfies
	\[
	K(t) \leq L(t) \leq K(4t)
	\]
	for all $t \geq 1$.
	By statement a),
	there exist constants $c_0 \in (0,1)$, 
	$C>0$, and $s_0 \geq \max\{ b, \, L(1)/c_0\}$ such that
	for all $\eta \in \mathbb{R}$ with $|\eta| \geq s_0$,
	\[
	\|R(i \eta,A)\| \leq  \frac{C}{L^{-1}(c_0|\eta|)}.
	\]
	Therefore, for all $s \geq s_0$, 
	\[
	\frac{1}{M(s)} = 
	\sup_{|\eta|\geq s}\|R(i\eta,A)\|  \leq 
	\sup_{|\eta|\geq s}	\frac{C}{L^{-1}(c_0|\eta|)}.
	\]
	Since $L^{-1}$ is non-decreasing, it follows that
	\[
	M(s) \geq  \frac{L^{-1}(c_0s) }{C}  
	\]
	for all $s \geq s_0$.
	This estimate implies that 
	for all $t > L^{-1}(c_0s_0)/C$ and $\varepsilon>0$,
	\begin{align*}
		M^{-1}(t) 
		\leq 
		\inf\{ 
		s \geq s_0 : L^{-1}(c_0s) \geq Ct
		\} 
		\leq \frac{L(Ct)}{c_0} + \varepsilon.
	\end{align*}
	Since $\varepsilon>0$ is arbitrary,
	we have
	\[
	M^{-1}(t) \leq \frac{L(Ct)}{c_0} \leq \frac{K(4Ct)}{c_0}
	\]
	for all $t > L^{-1}(c_0s_0)/C$.
	From this,
	the desired estimate \eqref{eq:M_inv_K} follows.
\end{proof}

We conclude this section by making a remark on
the conditions required in statement~a) of Theorem~\ref{thm:lower_bound}.
\begin{remark}
	Let $C_T\geq 1$ and $\omega \in \mathbb{R}$ satisfy
	$\|T(t)\| \leq C_T e^{\omega t}$ for all $t \geq 0$.
	Suppose that
	\begin{equation}
		\label{eq:liminf_K0}
		\liminf_{|\eta| \to \infty}\, |\eta|\, \|R(i\eta,A)\| > C_T.
	\end{equation}
	Then the piecewise linear 
	function $L$ defined as in the proof of statement~b)
	of Theorem~\ref{thm:lower_bound}
	satisfies
	the conditions in statement~a), namely,
	$t = O(L(t))$ as $t\to \infty$ and \eqref{eq:AT_estimate_K}
	with $K = L$.
	To see this, it suffices to show that 
	$1/t = O(\|AT(t)\|)$ as $t \downarrow 0$.
	Let $b \geq 0$ satisfy $\sigma(A) \cap i \mathbb{R} \subset (-ib,ib)$.
	If \eqref{eq:liminf_K0} holds, then
	there exist $s_0 > b$ and $\delta >0$ such that 
	\begin{equation}
		\label{eq:resol_lower_bound_remark}
		\|R(i \eta,A)\| \geq \frac{(1+2\delta)C_T}{|\eta|}
	\end{equation}
	for all $\eta \in \mathbb{R}$ satisfying $|\eta| \geq s_0$.
	There exists $t_0>0$ such that $\sup_{0\leq t \leq t_0} \|T(t)\|
	\leq (1+\delta) C_T$.
	By \eqref{eq:resol_for_lower_bound} with $z = i \eta$ and 
	\eqref{eq:resol_lower_bound_remark}, we have
	\begin{align*}
		\|R(i\eta,A)\| &\leq 
		\frac{1}{|\eta|}
		\left(
		\|AT(t)\| \, \|R(i\eta,A)\| + (1+\delta) C_T
		\right) + (1+\delta) C_Tt \\
		&\leq 
		\frac{\|AT(t)\| \, \|R(i\eta,A)\| }{|\eta|} + 
		\frac{1+\delta}{1+2\delta} \|R(i\eta,A)\| + 
		(1+\delta) C_Tt
	\end{align*}
	for all $t \in (0,t_0]$ and $\eta \in \mathbb{R}$ satisfying $|\eta| \geq s_0$.
	Hence,
	\begin{equation}
		\label{eq:AT_resol_remark}
		\frac{\delta}{1+2\delta} \leq 
		\frac{\|AT(t)\|  }{|\eta|} +
		\frac{(1+\delta) C_T t }{\|R(i\eta,A)\| }
	\end{equation}
	for all $t \in (0,t_0]$ and $\eta \in \mathbb{R}$ satisfying $|\eta| \geq s_0$.
	Take 
	$
	c \in \left( 
	0, \delta/(1+\delta)
	\right)
	$
	and $t \in (0, \min \{t_0,\, c/s_0  \}]$.
	Set
	\[
	\eta \coloneqq \frac{c}{t} \geq s_0,
	\] 
	By
	\eqref{eq:resol_lower_bound_remark} and 
	\eqref{eq:AT_resol_remark}, we obtain
	\[
	\frac{\delta}{1+2\delta}  \leq 
	\frac{t\|AT(t)\|}{c} +
	\frac{c(1+\delta) C_T}{\eta \|R(i\eta,A)\|} \leq 
	\frac{t\|AT(t)\|}{c} +\frac{c(1+\delta)}{1+2\delta}.
	\]
	Then
	\[
	\frac{\delta- c(1+\delta)}{1+2\delta} \leq \frac{t\|AT(t)\|}{c} .
	\]
	Thus, $1/t = O(\|AT(t)\|)$ as $t \downarrow 0$.
\end{remark}

\section{Growth of Hilbert space semigroups}
\label{sec:Hilbert}
For a Hilbert space semigroup
$(T(t))_{t \geq 0}$ with generator $A$, we transfer the rate of decay of $\|(i \eta - A)^{-1}\|$ as $|\eta| \to \infty$
to
the rate of growth of $\|AT(t)\|$ as $t \downarrow 0$. 
By employing 
functions of positive increase,
we establish an upper estimate for growth rates
on scales finer 
than the polynomial scales considered in \cite[Theorem~2.3]{Wakaiki2025_JFA}.
The proof is inspired by the techniques used for  \cite[Theorem~3.2]{Rozendaal2019}.
\begin{theorem}
	\label{thm:Hilbert_case}
	Let $A$ be the generator of a $C_0$-semigroup
	$(T(t))_{t \geq 0}$ on a Hilbert space 
	$X$ such that $\sigma(A) \cap i \mathbb{R}\subset (-i b, i b)$ for some $b \geq 0$.
	Suppose that $M \colon \mathbb{R}_+ \to (0,\infty)$
	is a non-decreasing continuous  function of positive increase such that 
	\begin{equation}
		\label{eq:resolvent_decay_Hilbert}
		\|R( i \eta,A)\| \leq \frac{1}{M(|\eta|)}
	\end{equation}
	for all $\eta \in \mathbb{R}$
	satisfying $|\eta| \geq b$.
	Then $(T(t))_{t \geq 0}$ is 
	immediately differentiable and satisfies
	\begin{equation}
		\label{eq:semigroup_growth_Hilbert}
		\|AT(t)\| = O \left(M^{-1} \left( \frac{1}{t} \right) \right)
		\qquad \text{as $t \downarrow 0$}.
	\end{equation}
\end{theorem}

To prove Theorem~\ref{thm:Hilbert_case},
we introduce some functions.
Throughout Section~\ref{sec:Hilbert},
let $\psi\colon \mathbb{R} \to \mathbb{R}_+$ be a Schwartz function such that 
$\|\psi\|_{L^{\infty}} = 1$, $\psi(\eta) = 1$ for $|\eta| \leq 1$, and 
$\supp \psi \subset [-2,2]$. Let $\phi$ 
be the inverse Fourier transform of
$\psi$, and
define $\phi_R(t) \coloneqq R\phi(Rt)$ for $t \in \mathbb{R}$ and $R >0$.
Let $\psi_R \coloneqq \mathcal{F}\phi_R$.
In contrast to the Banach space case discussed in Section~\ref{sec:Banach},
we do not consider a specific function for $\psi$ in this section.

For $n \in \mathbb{N}_0$ and
$t >0$, define $g_{n,t} \colon \mathbb{R}_+ \to \mathbb{R}$ by
\begin{equation}
	\label{eq:gnt}
	g_{n,t}(\tau) \coloneqq
	\begin{cases}
		\tau^n, & 0 \leq \tau \leq t, \\
		\tau^n - (\tau-t)^n, & \tau > t.
	\end{cases}
\end{equation}
Let $A$ be the generator of a bounded $C_0$-semigroup
$(T(t))_{t \geq 0}$ on a Hilbert space $X$.
For $n \in \mathbb{N}_0$,
$t >0$, and $x\in D(A)$,
we define $h_{n,t,x} \colon \mathbb{R} \to X$ by
\begin{equation}
	\label{eq:hntx}
	h_{n,t,x}(\tau) \coloneqq 
	\begin{cases}
		0, & \tau < 0, \\
		g_{n,t}(\tau)AT(\tau) x, & \tau \geq 0.
	\end{cases}
\end{equation}
Then
\begin{align*}
	AT(t)x =
	\frac{n+1}{t^{n+1}}
	\int^t_0 T(t-\tau) h_{n,t,x}(\tau)d\tau 
	= (n+1)(I_1  + I_2)
\end{align*}
for all $n \in \mathbb{N}_0$, $t >0$ and $x \in D(A)$,
where
\begin{align}
	I_1 &\coloneqq 
	\frac{1}{t^{n+1}}
	\int^t_0 T(t-\tau)  ( h_{n,t,x}(\tau) - (\phi_R * h_{n,t,x})(\tau))d\tau, \label{eq:I1_def}\\
	I_2&\coloneqq
	\frac{1}{t^{n+1}}
	\int^t_0 T(t-\tau)  (\phi_R * h_{n,t,x})(\tau)d\tau.\label{eq:I2_def}
\end{align}

To obtain an upper bound for $\|AT(t)\|$, 
we investigate the integrals $I_1$ and $I_2$
in Sections~\ref{sec:estimate_I1} and 
\ref{sec:estimate_I2}, respectively.
In Section~\ref{sec:proof_Hilbert}, we complete the proof of Theorem~\ref{thm:Hilbert_case}.
For the most part, we assume that 
$(T(t))_{t \geq 0}$ is exponentially stable, since
the non-exponential case follows from
the same argument as in 
Section~\ref{sec:non_exp_case_Banach}.

\begin{remark}
	\label{rem:M_O_s}
	In contrast to
	Theorem~\ref{thm:non_exp}, 
	we do not assume here
	that $M(s) = o(s)$ as $s \to \infty$.
	However, 
	we have $M(s) = O(s)$ as $s \to \infty$
	under the assumptions of Theorem~\ref{thm:Hilbert_case}.
	To see this, 
	let $A$ be a closed linear operator on a Banach space
	such that $\sigma(A) \cap i \mathbb{R}\subset (-i b, i b)$ 
	for some $b \geq 0$, and 
	let
	$M\colon \mathbb{R}_+ \to (0,\infty)$ satisfy
	\eqref{eq:resolvent_decay_Hilbert}
	for all $\eta \in \mathbb{R}$
	satisfying $|\eta| \geq b$. 
	Take
	$\eta_0 \in \mathbb{R}$ with $|\eta_0| \geq b$.
	For all
	$\eta \in \mathbb{R}$ satisfying $|\eta| \geq b$,
	the resolvent equation yields
	\[
	\|R( i \eta_0,A)\| \leq 
	\|R(i \eta,A)\| + |\eta- \eta_0| \,
	\|R(i \eta_0,A)\|\,\|R(i \eta,A)\|
	\]
	and hence
	\[
	\frac{1}{\|R(i \eta,A)\|}
	\leq 
	\frac{1}{\|R(i \eta_0,A)\|} + |\eta - \eta_0|.
	\]
	By \eqref{eq:resolvent_decay_Hilbert},
	\[
	M(|\eta|) \leq \frac{1}{\|R(i\eta,A)\|} = O(|\eta|)
	\qquad \text{as $|\eta| \to \infty$}
	\]
	is obtained.
\end{remark}

\subsection{Estimate for
	$I_1$}
\label{sec:estimate_I1}

Let $M \colon \mathbb{R}_+ \to (0,\infty)$
have positive increase. Then
there exist $\alpha >0$, $c \in (0,1]$, and $s_0 \geq 0$ such that 
\begin{equation}
	\label{eq:positive_inc_cond_Hilbert}
	\frac{M(\lambda s)}{M(s)} \geq c\lambda^\alpha
	\qquad \text{for all $\lambda \geq 1$ and $s \geq s_0$}.
\end{equation}
We estimate the integral $I_1$
by using the Fourier transform 
$\mathcal{F}(h_{n,t,x}- \phi_R * h_{n,t,x})$ and Plancherel's theorem.
The following lemma is used to 
estimate the $L^2$-norm of $\mathcal{F}(h_{n,t,x}- \phi_R * h_{n,t,x})$.
\begin{lemma}
	\label{lem:mn_upper_bound}
	Let $M \colon \mathbb{R}_+ \to (0,\infty)$ 
	satisfy 
	\eqref{eq:positive_inc_cond_Hilbert} for some 
	constants $\alpha >0$, $c \in (0,1]$, and $s_0 \geq 0$.
	Let $A$ be a closed linear operator on a Banach space $X$
	satisfying $\sigma(A)
	\cap i \mathbb{R} = \emptyset$.
	If 
	$\|R( i \eta,A)\| \leq M(|\eta|)^{-1}$ for all $\eta \in \mathbb{R}$,
	then
	there exist constants $C,R_0 >0$ such that 
	\[
	\sup_{\eta \in \mathbb{R}}
	\| (1 - \psi_R(\eta)) AR(i \eta,A)^n
	\| \leq 
	\frac{CR}{(cM(R))^n}
	\]
	for all $R\geq R_0$ and $n \in \mathbb{N}$ with $n\geq 1/\alpha$,
	where $\psi_R$
	is as in the paragraph following Theorem~\ref{thm:Hilbert_case}.
\end{lemma}
\begin{proof}
	Let $n \in \mathbb{N}$.
	We have
	\[
	AR(i \eta, A)^n = 
	-R(i \eta, A)^{n-1} + i \eta R(i \eta, A)^n.
	\]
	Hence, by assumption,
	\[
	\|AR(i \eta , A)^n\|  \leq \frac{1}{M(|\eta|)^{n-1}} + 
	\frac{|\eta|}{M(|\eta|)^n}
	\]
	for all $\eta \in \mathbb{R}$.
	Since $M(s)= O(s)$ as $s \to \infty$ (see Remark~\ref{rem:M_O_s}), 
	there exist $C,s_1>0$ such that 
	\begin{equation}
		\label{eq:AR_n_estimate}
		\|AR(i \eta , A)^n\| \leq \frac{C|\eta|}{M(|\eta|)^n}
	\end{equation}
	for all $\eta \in \mathbb{R}$ satisfying $|\eta| \geq s_1$.
	We also have $\psi_R(\eta) = \psi(\eta/R )$ for 
	all $\eta \in \mathbb{R}$ and $R>0$.
	Therefore, 
	\begin{equation}
		\label{eq:char_bound}
		|1 - \psi_R(\eta)| \leq 1 - \chi_{[-R,R]}(\eta) 
	\end{equation}
	for all $\eta \in \mathbb{R}$  and $R>0$, where
	$\chi_{[-R,R]}$ denotes the characteristic function of the interval
	$[-R,R]$.
	The estimates \eqref{eq:AR_n_estimate} and \eqref{eq:char_bound}
	yield
	\begin{equation}
		\label{eq:1_psi_AR_n}
		\|
		(1 - \psi_R(\eta))AR(i \eta , A)^n
		\| \leq 
		(1 - \chi_{[-R,R]}(\eta) )  \frac{C |\eta|}{M(|\eta|)^n}
	\end{equation}
	for all $\eta \in \mathbb{R}$ whenever $R \geq s_1$.

	If $|\eta|\geq R > s_0$, then
	\eqref{eq:positive_inc_cond_Hilbert} gives
	\[
	\frac{M(|\eta|)}{M(R)} \geq c \left(
	\frac{|\eta|}{R}
	\right)^{\alpha},
	\]
	and therefore,
	\[
	\frac{|\eta|}{M(|\eta|)^n} \leq 
	|\eta|
	\left(
	\frac{1}{c} 
	\left(
	\frac{R}{|\eta|}
	\right)^{\alpha}
	\frac{1}{M(R)}
	\right)^n =
	\frac{R^{n \alpha}}{|\eta|^{n \alpha - 1} (cM(R))^n}.
	\]
	Assume that  $n \alpha \geq 1$.
	Then $|\eta| \geq R$ implies 
	\[
	\frac{R^{n \alpha}}{|\eta|^{n \alpha - 1}} \leq R.
	\]
	This yields
	\begin{equation}
		\label{eq:1_chi_M}
		(1 - \chi_{[-R,R]}(\eta) )  \frac{|\eta|}{M(|\eta|)^n} \leq 
		\sup_{|\zeta| \geq  R}
		\frac{|\zeta|}{M(|\zeta|)^n} \leq 
		\frac{R}{(c M(R))^n} 
	\end{equation}
	for all $\eta \in \mathbb{R}$
	whenever $R >  s_0$. Combining the estimates \eqref{eq:1_psi_AR_n}
	and \eqref{eq:1_chi_M}, we obtain the desired conclusion.
\end{proof}

The following lemma gives an estimate for the integral 
$I_1$.
\begin{lemma}
	\label{lem:delta_phi_h_int_bound}
	Let $M \colon \mathbb{R}_+ \to (0,\infty)$ 
	satisfy 
	\eqref{eq:positive_inc_cond_Hilbert} for some 
	constants $\alpha >0$, $c \in (0,1]$, and $s_0 \geq 0$.
	Let $A$ be the generator of an exponentially stable
	$C_0$-semigroup
	$(T(t))_{t \geq 0}$
	on a Hilbert space $X$, and let $C_T  \coloneqq \sup_{t \geq 0} \|T(t)\|$.
	If $\|R( i \eta,A)\| \leq M(|\eta|)^{-1}$ for all $\eta \in \mathbb{R}$,
	then
	there exist constants $C,R_0 >0$ such that 
	\[
	\left\|
	\frac{1}{t^{n+1}}
	\int^t_0 T(t-\tau)  ( h_{n,t,x}(\tau) - (\phi_R * h_{n,t,x})(\tau))d\tau
	\right\| \leq
	\frac{n!CC_T^2R\|x\|}{(ctM(R))^n}
	\]
	for all 
	$n \in \mathbb{N}$ with $n\geq 1/\alpha$, $t >0$, $R \geq R_0$,
	and $x \in D(A)$, where 
	$\phi_R$ is as in the paragraph following Theorem~\ref{thm:Hilbert_case} and 
	$h_{n,t,x}$ is defined by \eqref{eq:hntx}.
\end{lemma}

\begin{proof}
	Let $n \in \mathbb{N}$, $t >0$, and $x \in D(A)$.
	The Cauchy--Schwarz inequality yields
	\begin{align}
		\label{eq:int_T_del_h_CS}
		&\left\|
		\frac{1}{t^{n+1}}
		\int^t_0 T(t-\tau)   ( h_{n,t,x}(\tau) - (\phi_R * h_{n,t,x})(\tau))d\tau
		\right\| \notag\\
		&\qquad \qquad \leq 
		\frac{C_T }{t^{n+1/2}} 
		\|h_{n,t,x}- \phi_R * h_{n,t,x}\|_{L^2(\mathbb{R})}.
	\end{align}
	By the exponential stability of $(T(t))_{t \geq 0}$, we have
	$h_{m,t,x} \in L^1(\mathbb{R},X) \cap L^2(\mathbb{R},X)$ for all $m \in \mathbb{N}_0$.
	Since 
	\[
	h_{n,t,x}(\tau) = n \int_0^{\tau} T(\tau - s) h_{n-1,t,x}(s)ds
	\] 
	for 
	all $\tau \geq 0$, Fubini's theorem shows that 
	\[
	\mathcal{F}(h_{n,t,x})(\eta) = nR(i \eta,A)  \mathcal{F}(h_{n-1,t,x}) (\eta)
	\]
	for all $\eta \in \mathbb{R}$.
	Repeating this argument, we derive
	\begin{equation}
		\label{eq:Fourier_induction}
		\mathcal{F}(h_{n,t,x})(\eta) = n! R(i \eta,A)^n \mathcal{F}(h_{0,t,x}) (\eta)
	\end{equation}
	for all $\eta \in \mathbb{R}$.
	Let $R>0$. By \eqref{eq:Fourier_induction},
	\begin{align*}
		(\mathcal{F} ( h_{n,t,x} - \phi_R * h_{n,t,x} ))(\eta) 
		&=
		(\mathcal{F} h_{n,t,x} )(\eta)  - 
		(\mathcal{F}\phi_R )(\eta) (\mathcal{F} h_{n,t,x})(\eta) \\
		&=
		(1 - \psi_R(\eta)) n! R(i \eta,A)^n \mathcal{F}(h_{0,t,x}) (\eta).
	\end{align*}
	Define $U_n(\eta) \coloneqq n! AR(i \eta,A)^n$
	for $\eta \in \mathbb{R}$ and
	\[
	h_{t,x}(\tau) \coloneqq 
	\begin{cases}
		0, & \tau < 0, \\
		g_{0,t}(\tau)T(\tau)x, & \tau \geq 0.
	\end{cases}
	\]
	Since $h_{0,t,x}(\tau) = Ah_{t,x}(\tau)$, we have
	\[
	\mathcal{F} ( h_{n,t,x} - \phi_R * h_{n,t,x}) =
	(1 - \psi_R) U_n \mathcal{F}h_{t,x}.
	\]
	
	Plancherel's theorem gives
	\begin{align*}
		\|h_{n,t,x} - \phi_R * h_{n,t,x}\|_{L^2(\mathbb{R})} &=
		\frac{1}{\sqrt{2\pi}}
		\| \mathcal{F}(h_{n,t,x} - \phi_R * h_{n,t,x})\|_{L^2(\mathbb{R})}\\ 
		&=
		\frac{1}{\sqrt{2\pi}}
		\| (1 - \psi_R) U_n \mathcal{F}h_{t,x}
		\|_{L^2(\mathbb{R})}
	\end{align*}
	and
	\begin{align*}
		\frac{1}{\sqrt{2\pi}}
		\| (1 - \psi_R) U_n \mathcal{F}h_{t,x}
		\|_{L^2(\mathbb{R})} &\leq 
		\frac{1}{\sqrt{2\pi}}\| (1 - \psi_R) U_n 
		\|_{L^\infty(\mathbb{R})} \|  \mathcal{F}h_{t,x}
		\|_{L^2(\mathbb{R})}\\ &=
		\| (1 - \psi_R) U_n 
		\|_{L^\infty(\mathbb{R})} \| h_{t,x}
		\|_{L^2(\mathbb{R})}.
	\end{align*}
	Hence,
	\begin{equation}
		\label{eq:del_phi_h}
		\|h_{n,t,x} - \phi_R * h_{n,t,x}\|_{L^2(\mathbb{R})} 
		\leq \| (1 - \psi_R) U_n 
		\|_{L^\infty(\mathbb{R})} \| h_{t,x}
		\|_{L^2(\mathbb{R})}.
	\end{equation}
	Moreover,
	\begin{equation}
		\label{eq:h_tx_bound}
		\| h_{t,x}
		\|_{L^2(\mathbb{R})}^2 =
		\int^t_0 \|T(\tau)x\|^2 d\tau \leq t C_T^2 \|x\|^2.
	\end{equation}
	Combining the estimates \eqref{eq:int_T_del_h_CS}, 
	\eqref{eq:del_phi_h}, and \eqref{eq:h_tx_bound}, we obtain
	\begin{align}
		&\left\|
		\frac{1}{t^{n+1}}
		\int^t_0 T(t-\tau)  (h_{n,t,x}(\tau)- (\phi_R * h_{n,t,x})(\tau))d\tau
		\right\| \notag \\
		&\qquad\qquad \leq
		\frac{C_T^2 \|x\|}{t^n}\| (1 - \psi_R) U_n 
		\|_{L^\infty(\mathbb{R})}.
		\label{eq:int_T_d_phi_h}
	\end{align}
	
	By Lemma~\ref{lem:mn_upper_bound}, 
	there exist $C,R_0 >0$ such that 
	\begin{equation}
		\label{eq:1_phi_L_inf}
		\| (1 - \psi_R) U_n 
		\|_{L^\infty(\mathbb{R})}
		\leq 
		\frac{n!CR}{(cM(R))^n}
	\end{equation}
	for all $R\geq R_0$ and $n \in \mathbb{N}$ with $n\geq 1/\alpha$.
	The desired conclusion follows from
	the estimates \eqref{eq:int_T_d_phi_h} and \eqref{eq:1_phi_L_inf}.
\end{proof}

\subsection{Estimate for
	$I_2$
}
\label{sec:estimate_I2}
Next, we derive an estimate for the integral 
$I_2$. To this end, we define 
\begin{equation}
	\label{eq:Phi_k_def}
	\Phi_1 \coloneqq |\phi'| \quad \text{and}\quad 
	\Phi_{k+1}(s) \coloneqq
	\begin{dcases}
		\int^s_{-\infty} \Phi_k(\sigma) d\sigma,& s < 0, \vspace{5pt}\\
		- \int^{\infty}_{s} \Phi_k(\sigma) d\sigma,& s \geq  0
	\end{dcases}
\end{equation}
for
$k \in \mathbb{N}$, where $\phi$ is 
as in the paragraph following Theorem~\ref{thm:Hilbert_case}.
\begin{lemma}
	\label{lem:phi_h_int_bound}
	Let $A$ be the generator of a bounded 
	$C_0$-semigroup
	$(T(t))_{t \geq 0}$
	on a Banach space $X$, and define 
	$C_T  \coloneqq \sup_{t \geq 0}\|T(t)\|$.
	For $k \in \mathbb{N}$, 
	let $\Phi_k$ be defined by \eqref{eq:Phi_k_def}.
	Then
	\begin{align}
		&\left\| 
		\frac{1}{t^{n+1}}
		\int^t_0 T(t-\tau)  (\phi_R * h_{n,t,x})(\tau)d\tau
		\right\| \notag \\
		&\qquad \qquad \leq
		2C_T^2R\|x\|
		\sum_{k=0}^{n} 
		\frac{n!}{(n-k+1)!}  \frac{\|\Phi_{k+1}\|_{L^1(\mathbb{R})}}{(Rt)^k}
		\label{eq:int_T_phi_h_lemma}
	\end{align}
	for all $n \in \mathbb{N}$, $t,R>0$, and
	$x \in D(A)$,
	where 
	$\phi_R$
	is as in the paragraph following Theorem~\ref{thm:Hilbert_case} and 
	$h_{n,t,x}$ is defined by \eqref{eq:hntx}.
\end{lemma}

\begin{proof}
	Let $n \in \mathbb{N}$, $t,R>0$, and
	$x \in D(A)$. We have
	\begin{align}
		\label{eq:int_T_phi_h}
		\left\| 
		\frac{1}{t^{n+1}}
		\int^t_0 T(t-\tau)  (\phi_R * h_{n,t,x})(\tau)d\tau
		\right\| 
		\leq
		\frac{C_T }{t^{n+1}}
		\int_0^t  \|(\phi_R * h_{n,t,x})(\tau)\|d\tau.
	\end{align}
	Define $H_{n,t,x}\colon \mathbb{R} \to X$ by
	\[
	H_{n,t,x}(s) \coloneqq \int^s_0 h_{n,t,x}(\sigma) d\sigma.
	\]
	Integration by parts yields
	\begin{align}
		(\phi_R * h_{n,t,x})(\tau)  &=
		R\int_0^{\infty} \phi(R(\tau-s)) h_{n,t,x}(s)ds \notag \\
		&= 
		R
		\int_0^\infty \phi'(R\tau-s) H_{n,t,x}\left( \frac{s}{R} \right) ds
		\label{eq:phi_hn_int}
	\end{align}
	for all $\tau \in [0,t]$.
	If $0 \leq s \leq t$, then
	\[
	H_{n,t,x}(s) = \int^s_0 \sigma^n AT(\sigma) x d\sigma =
	s^n T(s)x - n \int^s_0 \sigma^{n-1}T(\sigma) x d\sigma.
	\]
	This gives 
	\begin{equation}
		\label{eq:Hn_bound}
		\|H_{n,t,x}(s)\| \leq 
		2C_T s^n \|x\|.
	\end{equation}
	A routine calculation shows that 
	the same estimate holds for $s > t$. Hence by
	\eqref{eq:phi_hn_int},
	\begin{align}
		\label{eq:int_norm_phi_h}
		\int_0^t  \|(\phi_R * h_{n,t,x})(\tau)\|d\tau
		\leq 
		2C_T R\|x\|
		\int_0^t 
		\int_0^\infty |\phi'(R\tau-s)|\,  \left( 
		\frac{s}{R}
		\right)^n
		ds d\tau.
	\end{align}
	
	Fix $\tau \in [0,t]$ for a moment.
	To estimate the integral
	\[
	\int_0^\infty |\phi'(R\tau-s)|\,  \left( 
	\frac{s}{R}
	\right)^n
	ds,
	\]
	we define 
	\[
	\langle \Phi_k \rangle \coloneqq \int_{-\infty}^{\infty} \Phi_k(s)ds
	\]
	for $k \in \mathbb{N}$, where $\Phi_k$ is as in 
	\eqref{eq:Phi_k_def}.
	Then, for all $k \in \mathbb{N}$,
	\[
	\Phi_{k+1}' = \Phi_k - \langle \Phi_k \rangle \delta
	\]
	in the sense of distributions, where $\delta$ denotes the Dirac delta function 
	at zero. 
	Moreover, 
	for each 
	$k \in \mathbb{N}$ and
	$\ell \in \mathbb{N}_0$, there exists $C_{k,\ell}>0$ such that 
	for all $s \in \mathbb{R}$,
	\[
	|\Phi_k (s)| \leq \frac{C_{k,\ell}}{(1+|s|)^{\ell}}.
	\]
	Integrating by parts repeatedly, we derive
	\[
	\int_0^{\infty} |\phi'(\tau-s)| s^n ds =
	n! \int_{-\infty}^\tau \Phi_{n+1}(s) ds + 
	\sum_{k=0}^{n-1} 
	\frac{n!}{(n-k)!} \langle \Phi_{k+1} \rangle \tau^{n-k}.
	\]
	This yields
	\begin{equation}
		\label{eq:phi'_tau_int}
		\int_0^{\infty} |\phi'(R\tau-s)| \left( \frac{s}{R} \right)^n ds \leq 
		\sum_{k=0}^{n} 
		\frac{n!}{(n-k)!} \|\Phi_{k+1}\|_{L^1(\mathbb{R})} \frac{\tau^{n-k}}{R^k}.
	\end{equation}
	
	From \eqref{eq:phi'_tau_int}, it follows that
	\begin{align*}
		\int_0^t 
		\int_0^\infty |\phi'(R\tau-s)|\,  \left( 
		\frac{s}{R}
		\right)^n
		ds d\tau
		&\leq 
		\sum_{k=0}^{n} 
		\frac{n!}{(n-k+1)!} \|\Phi_{k+1}\|_{L^1(\mathbb{R})} \frac{t^{n-k+1}}{R^k}.
	\end{align*}
	Combining this with 
	\eqref{eq:int_T_phi_h} and \eqref{eq:int_norm_phi_h},
	we obtain the desired estimate \eqref{eq:int_T_phi_h_lemma}.
\end{proof}

\subsection{Estimate for $\|AT(t)\|$}
\label{sec:proof_Hilbert}
We are now in a position to prove 
Theorem~\ref{thm:Hilbert_case}.
\begin{proof}{Proof of Theorem~\ref{thm:Hilbert_case}}
	Assume 
	that the $C_0$-semigroup
	$(T(t))_{t \geq 0}$ is exponentially stable and that $b = 0$.
	Then the resolvent estimate \eqref{eq:resolvent_decay_Hilbert}
	holds for all $\eta \in \mathbb{R}$.
	Let constants $\alpha >0$, $c \in (0,1]$, and $s_0 \geq 0$
	satisfy
	\eqref{eq:positive_inc_cond_Hilbert}.
	Set 
	\begin{equation}
		\label{eq:n_def_Hilbert}
		n \coloneqq \min\{ m \in \mathbb{N} : m  \geq 1/\alpha \}. 
	\end{equation}
	Let the integrals $I_1$ and $I_2$ be defined by
	\eqref{eq:I1_def} and \eqref{eq:I2_def}, respectively.
	By Lemma~\ref{lem:delta_phi_h_int_bound},
	there exist $C,R_0 >0$ such that 
	\begin{equation}
		\label{eq:del_phi_int_bound}
		\|I_1\| \leq
		\frac{n!CC_T^2R\|x\|}{(ctM(R))^n}
	\end{equation}
	for all  $n \in \mathbb{N}$ with $n\geq 1/\alpha$,
	$t >0$, $R \geq R_0$, and $x \in D(A)$, where 
	$C_T  \coloneqq \sup_{t \geq 0} \|T(t)\|$.
	Moreover,
	Lemma~\ref{lem:phi_h_int_bound} gives
	\begin{equation}
		\label{eq:phi_int_bound}
		\|I_2\| \leq 
		2C_T^2R\|x\|
		\sum_{k=0}^{n} 
		\frac{n!}{(n-k+1)!} \frac{\|\Phi_{k+1}\|_{L^1(\mathbb{R})} }{(Rt)^k}
	\end{equation}
	for all 
	$n \in \mathbb{N}$, 
	$t,R>0$, and $x \in D(A)$, where
	$\Phi_{k}$ is defined by \eqref{eq:Phi_k_def}
	for $k\in \mathbb{N}$.
	
	Define 
	\[
	t_0 \coloneqq \frac{1}{c (M(R_0) + 1)}.
	\]
	Let $t \in (0,t_0]$ be arbitrary and
	set 
	\[
	R \coloneqq M^{-1}\left( \frac{1}{ct} \right).
	\]
	Then
	\begin{equation}
		\label{eq:R_lower_bound}
		R \geq 
		M^{-1}( M(R_0) + 1) 
		\geq R_0.
	\end{equation}
	By the property \eqref{eq:li_prop2}
	of left-inverses,
	\begin{equation}
		\label{eq:MR_lower_bound}
		M(R) = M\left(M^{-1} \left( \frac{1}{ct}\right) \right) = \frac{1}{ct}.
	\end{equation}
	This yields
	\begin{equation}
		\label{eq:R_tM_M_inv}
		\frac{R}{(ctM(R))^n} =
		M^{-1}\left( \frac{1}{ct} \right).
	\end{equation}
	Moreover,
	since $M(s) = O(s)$ as $s \to \infty$ (see Remark~\ref{rem:M_O_s}),
	there exists $C_1>0$ such that 
	$M(s) \leq C_1(s+1)$ for all $s \geq 0$.
	Combining 
	this with  \eqref{eq:R_lower_bound} and 
	\eqref{eq:MR_lower_bound}, we obtain
	\[
	Rt = \frac{R}{cM(R)} \geq \frac{R_0}{cC_1(R_0+1)} \eqqcolon c_1,
	\]
	where the constant $c_1$ does not depend on $t$.
	Hence,
	\begin{equation}
		\label{eq:R_sum_M_inv_sum}
		R
		\sum_{k=0}^{n} 
		\frac{n!}{(n-k+1)!}  \frac{\|\Phi_{k+1}\|_{L^1(\mathbb{R})}}{(Rt)^k}
		\leq 
		M^{-1}\left( \frac{1}{ct} \right)
		\sum_{k=0}^{n} 
		\frac{n!}{(n-k+1)!} \frac{\|\Phi_{k+1}\|_{L^1(\mathbb{R})}}{c_1^k}.
	\end{equation}
	Applying \eqref{eq:R_tM_M_inv}
	and \eqref{eq:R_sum_M_inv_sum} to
	the estimates \eqref{eq:del_phi_int_bound} and 
	\eqref{eq:phi_int_bound}, respectively, we see that 
	for all $x \in D(A)$,
	\[
	\|AT(t)x\| \leq (n+1)(\|I_1\| + \|I_2\|) \leq C_0 M^{-1}\left( \frac{1}{ct} \right)\|x\|,
	\]
	where
	\[
	C_0 \coloneqq  (n+1)!CC_T^2 + 
	2C_T^2
	\sum_{k=0}^{n} 
	\frac{ (n+1)!}{(n-k+1)!} \frac{\|\Phi_{k+1}\|_{L^1(\mathbb{R})}}{c_1^k}.
	\]
	Note that 
	$n$ depends only on the constant $\alpha$
	unlike in the Banach case; see 
	\eqref{eq:n_def_Banach} and \eqref{eq:n_def_Hilbert}.
	
	By \eqref{eq:M_poly_lower_bound} and 
	\eqref{eq:resolvent_decay_Hilbert}, we have
	$
	\|R(i\eta,A)\| = O(|\eta|^{-\alpha})
	$
	as $|\eta| \to \infty$.
	Corollary~\ref{coro:suff_id} shows 
	that $(T(t))_{t \geq 0}$ is immediately differentiable.
	Hence,
	$AT(t) \in \mathcal{L}(X)$ for all $t >0$.
	We conclude that the desired estimate
	\eqref{eq:semigroup_growth_Hilbert} holds,
	by using Lemma~\ref{lem:RSS22} in order
	to remove the constant $c$ from the estimate.
	
	Finally, in the case where $(T(t))_{t \geq 0}$ is not exponentially stable, an analogous argument to that in Section~\ref{sec:non_exp_case_Banach} applies. Therefore, we omit the details.
\end{proof}

\section{Necessity of positive increase}
\label{sec:nec_pi}

For a bounded $C_0$-semigroup $(T(t))_{t \geq 0}$,
the necessity of positive increase in the context of semi-uniform stability
was established in
\cite[Proposition~5.1]{Batty2016}
and 
\cite[Theorem~3.4]{Rozendaal2019}.
In this section, we show that the positive increase condition
in Theorem~\ref{thm:Hilbert_case} is also necessary,
provided that $M(s) = o(s)$ as $s \to \infty$ and that
the resolvent decay is governed
by the distance between $\sigma(A)$ and 
the set $\{i \eta: \eta \in \mathbb{R}
\text{~and~} |\eta| \geq s \}$.
The argument here 
is inspired by \cite[Theorem~3.4]{Rozendaal2019}.
\begin{proposition}
	\label{prop:necessity}
	Let $A$ be the generator of an
	immediately differentiable $C_0$-semigroup
	$(T(t))_{t \geq 0}$ on a Banach space $X$ such that $\sigma(A) \cap i \mathbb{R} \subset (-ib,ib)$ for some $b \geq 0$.
	Let $M \colon \mathbb{R}_+ \to (0,\infty)$
	be a non-decreasing continuous function 
	such that 
	$M(s) \to \infty$ as $s \to \infty$. Assume that
	there exist constants $\delta \in (0,1]$ and
	$\varepsilon \in (0,1)$
	such that 	for all $s \geq b$, 
	\begin{equation}
		\label{eq:M_ineq_delta_eps}
		M(s) \leq \delta \varepsilon s
	\end{equation}
	and 
	\begin{equation}
		\label{eq:M_R_ineq}
		\frac{\delta}{ M(s)} \leq
		\sup_{|\eta| \geq s} \frac{1}{\dist (i \eta , \sigma (A))}
		\leq 
		\sup_{|\eta| \geq s} \|R(i \eta,A)\| \leq \frac{1}{M(s)}.
	\end{equation}
	If there exists a constant $c >0$ such that 
	\begin{equation}
		\label{eq:AT_bound_nec}
		\|AT(t)\| = O \left(M^{-1} \left( \frac{1}{ct} \right) \right)
		\qquad \text{as $t \downarrow 0$},
	\end{equation}
	then $M$ has positive increase.
\end{proposition}

\begin{proof}
	Define $N\colon [b,\infty) \to (0,\infty)$ by
	\[
	N(s) \coloneqq \frac{1}{\displaystyle 
		\sup_{|\eta| \geq s} 1/\dist (i \eta , \sigma (A))} =
	\inf_{|\eta| \geq s} \dist (i \eta , \sigma (A)).
	\]
	We divide the proof into three steps.
	First, we find a point in $\sigma(A)$ satisfying
	certain inequalities. Next, we investigate the relation between
	$M^{-1}$ and $N$. Finally, using this relation, we show that $M$
	has positive increase.
	
	Step~1.
	Since
	$
	N(s) \leq M(s)/\delta
	$
	for all $s \geq b$ by \eqref{eq:M_R_ineq},
	the inequality \eqref{eq:M_ineq_delta_eps} yields
	\begin{equation}
		\label{eq:wN_bound_13}
		N(s) \leq \varepsilon s
	\end{equation}
	for all $s \geq b$. By Theorem~\ref{thm:spectral_prop},
	there exists $s_0 \geq b$ such that 
	\[
	\sup\{
	\re z :
	z \in \sigma(A) \text{~and~} |z| \geq (1-\varepsilon)^2s_0
	\} < 0.
	\]
	We will show that for all $s \geq s_0$,
	there exist $u_s,v_s \in \mathbb{R}$ with
	$u_s +iv_s  \in \sigma(A)$
	such that 
	\begin{subequations}
		\label{eq:alp_beta_bound}
		\begin{empheq}[left = {\empheqlbrace \,}, right = {}]{align}
			|u_s | &\leq (2-\varepsilon) N(s), \label{eq:alp_beta_bound1}\\
			(1-\varepsilon)^2s &\leq |u_s  + iv_s | , \quad \text{and}
			\label{eq:alp_beta_bound2} \\
			u_s &< 0. \label{eq:alp_beta_bound3}
		\end{empheq}
	\end{subequations}
	
	For each $s \geq  s_0$,
	there exist $\eta_s\in \mathbb{R}$ with $|\eta_s| \geq s$
	and $u_s,v_s \in \mathbb{R}$ with
	$u_s +iv_s  \in \sigma(A)$ such that
	\[
	N(s) \leq |i \eta_s - (u_s +iv_s )| \leq (2-\varepsilon) N(s).
	\]
	Hence, we obtain \eqref{eq:alp_beta_bound1}.
	Moreover, \eqref{eq:wN_bound_13} yields
	\begin{equation}
		\label{eq:2/3_bound}
		|i \eta_s  - (u_s +iv_s )| \leq \varepsilon(2-\varepsilon) s.
	\end{equation}
	By $|\eta_s | \geq s$, we have
	\begin{equation}
		\label{eq:s_us_vs_bound}
		s - |u_s  + i v_s | \leq |\eta_s | - |u_s  + i v_s | \leq 
		|i \eta_s - (u_s  + i v_s ) |.
	\end{equation}
	From \eqref{eq:2/3_bound} and \eqref{eq:s_us_vs_bound},
	it follows that 
	\eqref{eq:alp_beta_bound2} holds.
	Since $u_s +iv_s  \in \sigma(A)$ 
	satisfies 
	$|u_s+iv_s| \geq (1-\varepsilon)^2s_0$, we also obtain
	\eqref{eq:alp_beta_bound3}.

	Step~2.
	Let $s \geq s_0$ and take 
	$u_s,v_s \in \mathbb{R}$ with
	$u_s +iv_s  \in \sigma(A)$ such that 
	\eqref{eq:alp_beta_bound1}--\eqref{eq:alp_beta_bound3} hold.
	By \cite[Lemma~2.4.6]{Pazy1983}, we obtain
	$z e^{z t} \in \sigma(AT(t))$
	for all $z \in \sigma(A)$ and $t >0$.
	Therefore,
	\[
	(u_s + i v_s) e^{(u_s+iv_s)t} \in \sigma(AT(t))
	\]
	for all $t >0$.
	Since $AT(t)$ is a bounded operator on $X$ for all $t>0$,
	we have
	\[
	|u_s+iv_s| e^{u_s t} \leq 
	\sup_{z \in \sigma(AT(t))} |z|
	\leq
	\|AT(t)\|
	\]
	for all $t >0$.
	By assumption~\eqref{eq:AT_bound_nec},
	there exist $C>0$ and $t_0 \in (0,1/(cM(0))]$ such that 
	for all $t \in (0,t_0]$,
	\begin{equation}
		\label{eq:M_inv_bound}
		|u_s+iv_s| e^{u_s t} 
		\leq 
		CM^{-1} \left(
		\frac{1}{ct}
		\right).
	\end{equation}
	
	Let $0 < t \leq t_0$.
	By \eqref{eq:alp_beta_bound1} and \eqref{eq:alp_beta_bound3}, we obtain
	\begin{equation}
		\label{eq:at_wideNt}
		-u_s t = |u_s| t \leq (2-\varepsilon) N(s)t.
	\end{equation}
	Moreover, \eqref{eq:alp_beta_bound2} yields
	\begin{equation}
		\label{eq:ab_exp_lower}
		(1-\varepsilon)^2se^{u_s t}  \leq 
		|u_s + iv_s | e^{u_s t}.
	\end{equation}
	Combining \eqref{eq:ab_exp_lower} with \eqref{eq:M_inv_bound}, we obtain
	\begin{equation}
		\label{eq:eu_s_M_inv_bound}
		(1-\varepsilon)^2se^{u_s t}  \leq 
		CM^{-1} \left(
		\frac{1}{ct}
		\right).
	\end{equation}
	By \eqref{eq:at_wideNt} and 
	\eqref{eq:eu_s_M_inv_bound}, 
	\begin{equation}
		\label{eq:log_tildeN}
		\log\left(
		\frac{	(1-\varepsilon)^2s}{C M^{-1}(1/(ct))}
		\right) \leq 
		-u_st \leq 
		(2-\varepsilon) N(s)t.
	\end{equation}
	
	Step~3.
	Since $M(s) \to \infty$ as $s \to \infty$ and since
	$N(s) \geq M(s)$ for all $s \geq b$ by \eqref{eq:M_R_ineq},
	there exists $s_1 \geq b$ such that
	$N(s) \geq 1/(c\delta t_0)$ 
	for all $s \geq s_1$.
	Let $\lambda \geq 1$ and 
	$s \geq \max\{s_0, \, s_1\}$. Set
	\[
	t \coloneqq \frac{1}{c \delta N (s)}.
	\]
	Then $t \leq t_0$. 
	By \eqref{eq:log_tildeN}, we obtain
	\begin{align}
		\label{eq:N_frac_log}
		\frac{2-\varepsilon}{c\delta} \frac{N(\lambda s)}{N(s)}
		\geq 
		\log\left(
		\frac{(1-\varepsilon)^2\lambda s}{CM^{-1} (\delta N(s))}
		\right).
	\end{align}
	Since $\delta N(s) \leq M(s)$ by \eqref{eq:M_R_ineq} 
	and since $M^{-1}$
	is non-decreasing, 
	the property \eqref{eq:li_prop1} of left-inverses gives
	\[
	M^{-1} (\delta N(s)) \leq 
	M^{-1} (M(s)) \leq s.
	\]
	Applying this to \eqref{eq:N_frac_log}, we obtain
	\begin{equation}
		\label{eq:N_lower_bound}
		\frac{2-\varepsilon}{c\delta} \frac{N(\lambda s)}{N(s)}
		\geq
		\log \left(
		\frac{(1-\varepsilon)^2\lambda}{C}
		\right).
	\end{equation}
	Since $M(s) \leq N(s)$ and 
	$N(\lambda s) \leq M(\lambda s)/\delta$ 
	by \eqref{eq:M_R_ineq},  the estimate \eqref{eq:N_lower_bound} yields
	\[
	\frac{M(\lambda s)}{M(s)}
	\geq
	\frac{c\delta^2}{2-\varepsilon}
	\log \left(
	\frac{(1-\varepsilon)^2\lambda}{C}
	\right).
	\]
	The term on the right-hand side above
	is greater than $1$ for $\lambda$ sufficiently large.
	Thus, $M$ has positive increase by 
	\cite[Lemma~2.1]{Rozendaal2019}.
\end{proof}

\begin{remark}
	\label{rem:M_s_infty_prop}
	By \eqref{eq:M_R_ineq} and \eqref{eq:alp_beta_bound1}--\eqref{eq:alp_beta_bound3}, 
	we can obtain
	$M(s) \to \infty$ as $s \to \infty$. 
	Therefore, the assumption that $M(s) \to \infty$ as $s \to \infty$ can be omitted from Proposition~\ref{prop:necessity}.
	Indeed, let $u_s,v_s \in \mathbb{R}$ with
	$u_s +iv_s  \in \sigma(A)$
	satisfy
	\eqref{eq:alp_beta_bound1}--\eqref{eq:alp_beta_bound3}.
	By Theorem~\ref{thm:spectral_prop},
	there exist constants $p,q >0$ such that 
	$|\im z| \leq pe^{-q \re z}$ for all $z 
	\in \sigma(A)$.
	Hence, \eqref{eq:alp_beta_bound2} 
	and \eqref{eq:alp_beta_bound3} give
	\[
	(1-\varepsilon)^2 s \leq |u_s| + p e^{q|u_s|}.
	\]
	Since 
	\[
	\inf\{ \xi \geq 0 :  (1-\varepsilon)^2s \leq 
	\xi + pe^{q\xi} \} \to \infty
	\qquad \text{as~}s \to \infty,
	\]
	it follows that $|u_s| \to \infty$ as $s \to \infty$.
	Therefore, \eqref{eq:M_R_ineq} and \eqref{eq:alp_beta_bound1} imply that 
	$M(s)\to \infty$ as $s \to \infty$.
\end{remark}

\section{Growth of quasi-multiplication semigroups}
\label{sec:multiplication}
In this section, we present a sharper result
for the following special class of $C_0$-semigroups:
We say that an immediately differentiable
$C_0$-semigroup $(T(t))_{t \geq 0}$
on a Banach space with generator $A$ is of {\em 
	quasi-multiplication type} if
\begin{equation}
	\label{eq:multi_AT}
	\|AT(t)\| = \sup_{z \in \sigma(A)} |z| e^{t \re z}
\end{equation}
for all $t > 0$ and
\begin{equation}
	\label{eq:multi_R}
	\|R(z,A)\| = \sup_{\lambda \in \sigma(A)} \frac{1}{|z-\lambda|}
\end{equation}
for all $z \in \varrho(A) \cap i \mathbb{R}$.
The term
{\em quasi-multiplication $C_0$-semigroup}
was introduced in \cite[Section~5.1]{Batty2016} to study semi-uniform stability,
and  was
subsequently adopted in \cite[Section~4]{Rozendaal2019}.
Our definition of $C_0$-semigroups of 
quasi-multiplication type is
tailored to the analysis of the growth rate of $\|AT(t)\|$
as $t \downarrow 0$, and hence differs slightly from
the definitions given in \cite{Batty2016,Rozendaal2019}.

Given a constant $b >0$ and 
a non-decreasing function $M \colon [b,\infty) \to (0,\infty)$, define the function $M_{\inf}$ by
\begin{equation}
	\label{eq:Minf_def}
	M_{\inf} (s) \coloneqq \inf_{\lambda >1} \frac{M(\lambda s)}{\log \lambda},\quad s \geq b.
\end{equation}
To avoid having $M_{\inf}(0) = 0$,
here we do not consider the case $b=0$
unlike in the previous sections. 
The following theorem gives lower and upper estimates for
$\|AT(t)\|$.
We use statement a) to guarantee the existence of the left-inverse $M_{\inf}^{-1}$, and
the proof of statement~b) builds on the argument given in
\cite[Theorem~4.4]{Rozendaal2019}.
\begin{theorem}
	\label{thm:mul_lower_bound}
	Let $A$ be the generator of an immediately differentiable
	$C_0$-semigroup $(T(t))_{t \geq 0}$ of quasi-multiplication type
	on a Banach space $X$.
	Let $b > 0$ satisfy $\sigma(A) \cap i \mathbb{R} \subset 
	(-ib,ib)$, and
	define $M \colon [b,\infty) \to (0,\infty)$ by
	\begin{equation}
		\label{eq:M_def_multiplication}
		M(s) \coloneqq \frac{1}{\sup_{|\eta| \geq s} \|R(i \eta, A)\|}.
	\end{equation}
	If $M(s) = o(s)$  as $s \to \infty$,
	then
	the following statements hold:
	\begin{enumerate}
		\renewcommand{\labelenumi}{\rm{\alph{enumi})}}
		\item
		The function
		$M_{\inf}$ defined by \eqref{eq:Minf_def} is  non-decreasing 
		and continuous on $[b, \infty)$.
		Moreover,
		$M_{\inf}(s) \in (0,\infty)$ for all $s \in [b,\infty)$ and
		$M_{\inf}(s) \to \infty$ as $s \to \infty$.
		\item 
		For all $\varepsilon \in(0,1)$, there exists a constant $t_0 \in (
		0,1/M_{\inf}(b)
		]$ such that
		\[
		(1-\varepsilon)M_{\inf}^{-1}\left( \frac{1}{t}\right) \leq 
		\|AT(t)\| \leq (1+\varepsilon)M_{\inf}^{-1}\left( \frac{1}{t}\right)
		\]
		for all $t \in (0,t_0]$.
	\end{enumerate}
\end{theorem}

First, we investigate the properties of $M_{\inf}$
in Section~\ref{sec:prop_Minf}, which
covers most of the proof of statement a) of 
Theorem~\ref{thm:mul_lower_bound}. Next, in
Section~\ref{sec:estimate_AT_qm}, we derive 
the upper and lower bounds for $\|AT(t)\|$ in statement b) 
to complete the proof 
of Theorem~\ref{thm:mul_lower_bound}. 
\subsection{Properties of $M_{\inf}$}
\label{sec:prop_Minf}
In the following lemma, we collect the basic properties of $M_{\inf}$ that will be used 
in the proof of Theorem~\ref{thm:mul_lower_bound}.
\begin{lemma}
	\label{lem:M_inf_prop1}
	Let $b > 0$ and 
	let $M \colon [b,\infty) \to (0,\infty)$ 
	be a non-decreasing function.
	Then
	the function 
	$M_{\inf}$ defined by \eqref{eq:Minf_def}
	satisfies the following properties:
	\begin{enumerate}
		\renewcommand{\labelenumi}{\rm{\alph{enumi})}}
		\item $M_{\inf}$ is non-decreasing.
		\item 
		$M_{\inf}(s) \to \infty$ as $s \to \infty$ if and only if
		$\log s = o(M(s))$ as  $s \to \infty$.
		
		\item 
		If $\log s = O(M(s))$ as $s \to \infty$, then
		$M_{\inf}(s) \in (0,\infty)$ for all $s \geq b$.
		
		\item	If $M$ is right-continuous, then
		$M_{\inf}$ is also right-continuous.
		\item
		If $M$ is continuous and satisfies
		$\log s = o(M(s))$ as $s \to \infty$, then
		$M_{\inf}$ is also continuous.
	\end{enumerate}
\end{lemma}
\begin{proof}
	a) Let $s_2 \geq s_1 \geq b$.
	Since $M$ is non-decreasing, we have
	\[
	\frac{M(\lambda s_1)}{\log \lambda} \leq 
	\frac{M(\lambda s_2)}{\log \lambda}
	\]
	for all $\lambda > 1$.
	Hence,
	\[
	M_{\inf}(s_1) = \inf_{\lambda > 1}\frac{M(\lambda s_1)}{\log \lambda} \leq 
	\frac{M(\mu s_2)}{\log \mu}
	\]
	for all $\mu > 1$.
	This gives
	$M_{\inf}(s_1)  \leq M_{\inf}(s_2)$. Therefore,
	$M_{\inf}$ is non-decreasing.
	
	b)
	Suppose that $M_{\inf}(s) \to \infty$ as $s \to \infty$.
	Let $\varepsilon >0$ be arbitrary.
	Then there exists $s_0 \geq b$ such that 
	\[
	\inf_{\lambda>1} \frac{M(\lambda s_0)}{\log \lambda}=
	M_{\inf}(s_0) 
	\geq \frac{2}{\varepsilon}.
	\]
	Set $s_1 \coloneqq \max\{s_0,\,2\}$.
	For all $\lambda \geq s_1$,
	\[
	\frac{M(\lambda s_1)}{\log (\lambda s_1)} \geq 
	\frac{M(\lambda s_0)}{2\log \lambda } \geq \frac{1}{\varepsilon}.
	\]
	For all $s \geq  s_1^2$, we let $\lambda = s/s_1$ and derive
	\[
	\log s \leq \varepsilon M(s).
	\]
	Since $\varepsilon>0$ is arbitrary, 
	it follows that $\log s = o(M(s))$ as $s \to \infty$.
	
	Conversely,
	suppose that $\log s = o(M(s))$ as  $s \to \infty$.
	Let $b_1 \coloneqq \max \{b,\,2 \}$ and let 
	$f \colon [b_1,\infty) \to (0,\infty)$ be defined by
	\[
	f(s) \coloneqq \frac{\log s}{M(s)}.
	\]
	Then $f(s) \to 0$
	as $s \to \infty$.
	Let $\varepsilon >0$ be arbitrary.
	There exists $s_1 \geq  b_1$ such that 
	$f(s) \leq \varepsilon$ for all $s \geq s_1$. 
	For all $\lambda >1$ and $s \geq s_1$, 
	\[
	\frac{M(\lambda s)}{\log \lambda}
	=
	\frac{\log(\lambda s) / f(\lambda s)}{\log \lambda}
	\geq \frac{1}{f(\lambda s)} \geq \frac{1}{\varepsilon}.
	\]
	This implies that 
	\[
	M_{\inf} (s) \geq \frac{1}{\varepsilon}
	\]
	for all $s \geq s_1$.
	Since $\varepsilon>0$ is arbitrary,
	$M_{\inf}(s) \to \infty$ as $s \to \infty$.

	c)
	Assume, to reach a contradiction, that $M_{\inf}(s_0) = 0$ for some $s_0 \geq b$.
	Then there exists a sequence $(\lambda_n)_{n \in \mathbb{N}}$ in $(1,\infty)$
	such that 
	\begin{equation}
		\label{eq:M_log_con}
		\lim_{n \to \infty}
		\frac{M(\lambda_n s_0)}{\log\lambda_n} =0.
	\end{equation}
	We have $\sup_{n \in \mathbb{N}} \lambda_n = \infty$,
	and hence
	there exists a subsequence $(\lambda_{n_k})_{k \in \mathbb{N}}$
	such that $\lambda_{n_k}\to \infty$ as $k \to \infty$.
	On the other hand, since $\log s = O(M(s))$ as $s \to \infty$,
	there exist $C>0$ and $s_1 \geq \max\{b,\, 2\}$ such that 
	$\log s \leq CM(s)$ for all $s \geq s_1$.
	Hence, for all $k \in \mathbb{N}$ satisfying $\lambda_{n_k} s_0 \geq s_1$,
	\[
	\frac{M(\lambda_{n_k} s_0)}{\log \lambda_{n_k}} \geq 
	\frac{\log \lambda_{n_k}+ \log s_0}{C\log \lambda_{n_k}},
	\]
	and the term on the right-hand side converges to $1/C$
	as $k \to \infty$. This contradicts \eqref{eq:M_log_con}.
	Thus, $M_{\inf}(s) \in (0,\infty)$ for all $s \geq b$.
	
	d)
	Let $\sigma \geq s \geq b$ and $\varepsilon >0$.
	There exists $\lambda_0 >1$ such that 
	\begin{equation}
		\label{eq:Minf_bound1}
		\frac{M(\lambda_0 s)}{\log \lambda_0} - \frac{\varepsilon}{2}
		\leq \inf_{\lambda >1} \frac{M(\lambda s)}{\log \lambda}
		=
		M_{\inf} (s).
	\end{equation}
	We also have
	\begin{equation}
		\label{eq:Minf_bound2}
		M_{\inf} (\sigma) = 
		\inf_{\lambda >1} \frac{M(\lambda \sigma)}{\log \lambda}
		\leq \frac{M(\lambda_0 \sigma)}{\log \lambda_0}.
	\end{equation}
	Since $M$ is right-continuous, 
	there exists $\delta>0$
	such that 
	\begin{equation}
		\label{eq:Minf_bound3}
		|M(\lambda_0 \sigma) - M(\lambda_0 s)| \leq \frac{\varepsilon \log \lambda_0}{2}
	\end{equation}
	whenever $\sigma \in [s,s+\delta )$.
	Combining the estimates \eqref{eq:Minf_bound1}--\eqref{eq:Minf_bound3} 
	with the non-decreasing property of $M_{\inf}$, we see 
	that if $\sigma \in [s,s+\delta )$, then
	\[
	\frac{M(\lambda_0 s)}{\log \lambda_0} - 
	\frac{\varepsilon}{2} \leq 
	M_{\inf}(s) \leq 
	M_{\inf}(\sigma) \leq 
	\frac{M(\lambda_0 \sigma)}{\log \lambda_0} 
	\leq 
	\frac{M(\lambda_0 s)}{\log \lambda_0} + 
	\frac{\varepsilon}{2}.
	\]
	Therefore, $|M_{\inf}(\sigma) - M_{\inf}(s)| \leq \varepsilon$.
	This implies that $M_{\inf}$ is right-continuous.
	
	e)
	By statement d), it is enough
	to prove that $M_{\inf}$ is left-continuous.
	Let $s >b$.
	We first show that
	there exist $\lambda_1,\lambda_2 >1$ such that 
	for all $\sigma \in  [b,s]$,
	\begin{equation}
		\label{eq:Minf_min}
		M_{\inf}(\sigma) = \min_{\lambda_1 \leq \lambda \leq \lambda_2}
		\frac{M(\lambda \sigma)}{\log \lambda} .
	\end{equation}
	
	Define 
	\[
	\lambda_1 \coloneqq  \exp \left(
	\frac{M(b)}{M_{\inf}(s)+1}
	\right).
	\]
	Then, for all $\lambda \in (1,\lambda_1]$,
	\begin{equation}
		\label{eq:M_lower_lam1_1}
		\frac{M(b)}{\log \lambda} \geq M_{\inf}(s)+1.
	\end{equation}
	For all $\sigma \in [b,s]$ and $\lambda > 1$,
	we also have
	\begin{equation}
		\label{eq:M_lower_lam1_2}
		\frac{M(\lambda \sigma)}{\log \lambda}
		\geq \frac{M(b)}{\log \lambda}.
	\end{equation}
	By \eqref{eq:M_lower_lam1_1} and 
	\eqref{eq:M_lower_lam1_2},
	\begin{equation}
		\label{eq:inf_lambda1}
		\inf_{1 < \lambda \leq \lambda_1}
		\frac{M(\lambda \sigma)}{\log \lambda} 
		\geq  M_{\inf}(s) + 1\quad 
		\text{for all $\sigma \in [b,s]$.}
	\end{equation}
	In addition, 
	if we define $g\colon (1,\infty) \to (0,\infty)$ by
	\begin{equation}
		\label{eq:g_def_Minf}
		g(\lambda) \coloneqq \frac{M(\lambda b)}{\log \lambda },
	\end{equation}
	then 
	\begin{equation}
		\label{eq:M_log_g1}
		\frac{M(\lambda \sigma)}{\log \lambda}
		\geq 
		\frac{M(\lambda b)}{\log \lambda}
		=
		g(\lambda)
	\end{equation}
	for all $\sigma \in [b,s]$ and $\lambda > 1$.
	By assumption,
	we have $\log \lambda = o (M(\lambda b))$ as $\lambda \to \infty$. Hence, by the definition \eqref{eq:g_def_Minf}
	of $g$,
	there exists $\lambda_2 >\lambda_1$ such that 
	\begin{equation}
		\label{eq:M_log_g2}
		g(\lambda ) \geq M_{\inf}(s)+ 1
	\end{equation}
	for all $\lambda \geq \lambda_2$.
	By \eqref{eq:M_log_g1} and \eqref{eq:M_log_g2}, 
	\begin{equation}
		\label{eq:inf_lambda2}
		\inf_{\lambda \geq \lambda_2}
		\frac{M(\lambda \sigma)}{\log \lambda} \geq 
		M_{\inf}(s) +1\quad 
		\text{for all $\sigma \in [b,s]$.}
	\end{equation}
	Since $M_{\inf}(\sigma) \leq M_{\inf}(s)$
	for all $\sigma \in [b,s]$ by statement a),
	the estimates \eqref{eq:inf_lambda1}
	and 
	\eqref{eq:inf_lambda2} yield
	\eqref{eq:Minf_min}
	for all $\sigma \in [b,s]$.

	Let $\varepsilon >0$ be arbitrary.
	Since $M$ is uniformly continuous on $[\lambda_1b,\lambda_2s]$ by assumption,
	there exists $\delta>0$ such that 
	for all $t,\tau \in [\lambda_1b,\lambda_2s]$ 
	satisfying $|t-\tau| <  \delta$,
	\begin{equation}
		\label{eq:M_uni_cont}
		|M(t) - M(\tau)| \leq \varepsilon
		\log \lambda_1.
	\end{equation}
	Let $\sigma \in [b,s]$ satisfy $s-\sigma < \delta / \lambda_2$. Then \eqref{eq:M_uni_cont} yields
	\begin{equation}
		\label{eq:uniform_cont}
		\frac{M(\lambda s) - M(\lambda \sigma)}{\log \lambda}
		\leq \varepsilon \frac{\log \lambda_1}{\log \lambda}
		\leq \varepsilon
	\end{equation}
	for all $\lambda \in [\lambda_1,\lambda_2]$.
	By \eqref{eq:Minf_min},
	there exists $\lambda_{\sigma}
	\in [\lambda_1,\lambda_2]$ such that 
	\begin{equation}
		\label{eq:Minf_point}
		M_{\inf}(\sigma) = 
		\frac{M(\lambda_{\sigma} \sigma)}{\log \lambda_{\sigma}}.
	\end{equation}
	From \eqref{eq:uniform_cont} and \eqref{eq:Minf_point},
	it follows that
	\[
	M_{\inf}(s)
	= \inf_{\lambda > 1} \frac{M(\lambda s)}{\log \lambda}
	\leq 
	\frac{M(\lambda_{\sigma} s)}{\log \lambda_{\sigma}}
	\leq 
	\frac{M(\lambda_{\sigma} \sigma)}{\log \lambda_{\sigma}} + \varepsilon  =
	M_{\inf}(\sigma) + \varepsilon.
	\]
	Since $M_{\inf}$ is non-decreasing by a),
	we have
	$
	|M_{\inf}(s) - M_{\inf}(\sigma)| \leq \varepsilon.
	$
	Thus, $M_{\inf}$ is left-continuous.
\end{proof}

\subsection{Estimate for $\|AT(t)\|$}
\label{sec:estimate_AT_qm}

The argument to derive the lower bound for $\|AT(t)\|$ in statement b)
of Theorem~\ref{thm:mul_lower_bound}
is similar to that in Proposition~\ref{prop:necessity}.
To obtain a sharper estimate,
the following lemma 
will be used.

\begin{lemma}
	\label{lem:inf_dist_point}
	Let $A$ be a closed linear operator on a Banach space $X$ such that $\sigma(A)$ is non-empty and 
	$\sigma(A) \cap i \mathbb{R} \subset (-ib, ib)$ 
	for some $b \geq 0$.
	If
	\begin{equation}
		\label{eq:inf_dist_lemma}
		\inf_{|\eta| \geq s} \dist(i \eta , \sigma(A)) \to \infty
		\qquad \text{as $s \to \infty$},
	\end{equation}
	then
	for all $s \geq b$,
	there exist $\eta_0 \in \mathbb{R}$ with $|\eta_0| \geq s$
	and $z_0 \in \sigma(A)$ such that 
	\begin{equation}
		\label{eq:inf_dist_bound}
		|i \eta_0  - z_0| = \inf_{|\eta| \geq s}
		\dist(i \eta, \sigma (A)).
	\end{equation}
\end{lemma}
\begin{proof}
	First, we show that 
	for all $s \geq b$, there exists $s_0 \geq s$ such that 
	\begin{equation}
		\label{eq:inf_dist_bound1}
		\inf_{|\eta| \geq s}
		\dist(i \eta, \sigma (A)) =
		\inf_{s \leq |\eta| \leq s_0}
		\dist(i \eta, \sigma (A)).
	\end{equation}
	Assume for a contradiction that 
	there exists $s_1 \geq b$ such that for all $s \geq s_1$,
	\[
	\inf_{|\eta| \geq s_1}
	\dist(i \eta, \sigma (A)) <
	\inf_{s_1 \leq |\eta| \leq s}
	\dist(i \eta, \sigma (A)).
	\]
	Then, for all $n \in \mathbb{N}$ with $n \geq s_1$,
	there exists $\eta_n \in \mathbb{R}$ with $|\eta_n| \geq s_1$
	such that 
	\begin{equation}
		\label{eq:inf_dist_contradiction}
		\dist (i \eta_n, \sigma(A)) < \inf_{s_1 \leq |\eta| \leq n}
		\dist(i \eta, \sigma (A)).
	\end{equation}
	This also implies that $|\eta_n| > n$.
	By  assumption~\eqref{eq:inf_dist_lemma},
	$\dist (i \eta_n, \sigma(A)) \to \infty$ as $n \to \infty$.
	On the other hand, for all $n \in \mathbb{N}$ satisfying 
	$n \geq s_1$,
	\[
	\inf_{s_1 \leq |\eta| \leq n}
	\dist(i \eta, \sigma (A)) \leq 
	\dist(i s_1, \sigma (A)) < \infty.
	\]
	These contradict the inequality \eqref{eq:inf_dist_contradiction}.
	
	Let $s_0 \geq s \geq  b$.
	There exists $R>0$ such that
	\begin{equation}
		\label{eq:inf_dist_bound2}
		\inf_{s \leq |\eta| \leq s_0}\dist(i \eta, \sigma (A)) = 
		\inf_{s \leq |\eta| \leq s_0}\dist(i \eta, \sigma (A) \cap \overline{\mathbb{D}}_R),
	\end{equation}
	where $\overline{\mathbb{D}}_R$ denotes the closed ball
	with radius $R$ centered at the origin.
	Since the sets $\{\eta \in \mathbb{R}: s \leq |\eta| \leq s_0\}$
	and $\sigma (A)\cap \overline{\mathbb{D}}_R$ are 
	compact, there exist $\eta_0 \in \mathbb{R}$ with $s \leq |\eta_0| \leq s_0$
	and $z_0 \in \sigma (A)\cap \overline{\mathbb{D}}_R$ such that 
	\[
	|i \eta_0  - z_0| = \inf_{s \leq |\eta| \leq s_0}\dist(i \eta, \sigma (A) \cap \overline{\mathbb{D}}_R).
	\]
	Combining this with \eqref{eq:inf_dist_bound1}
	and \eqref{eq:inf_dist_bound2}, we conclude that 
	\eqref{eq:inf_dist_bound} holds.
\end{proof}

Let $A$ be the generator of an immediately differentiable
$C_0$-semigroup $(T(t))_{t \geq 0}$ 
on a Banach space.
For $\varepsilon \in (0,1)$, we define $K_{\varepsilon}\colon
(0,\infty) \to \mathbb{R}_+$ by
\begin{equation}
	\label{eq:K_eps_def}
	K_{\varepsilon}(t) \coloneqq 
	\frac{\|AT(1/t)\| }{1 - \varepsilon}.
\end{equation}
Next, we present several properties of $K_{\varepsilon}$,
which are crucial for 
establishing 
the lower estimate in statement b) of 
Theorem~\ref{thm:mul_lower_bound}.
Note that under the assumptions of Theorem~\ref{thm:mul_lower_bound},
the function $M$
defined by \eqref{eq:M_def_multiplication}
satisfies $M(s) \to \infty$ as $s \to \infty$.
Indeed, it follows from \eqref{eq:multi_R} that
\[
\frac{1}{M(s)} = \sup_{|\eta| \geq s} \|R(i \eta,A)\| =
\sup_{|\eta| \geq s} \frac{1}{\dist (i \eta , \sigma (A))}
\]
for all $s \geq b$, and hence 
$M(s) \to \infty$ as $s \to \infty$ by
the arguments
in Remark~\ref{rem:M_s_infty_prop}.

\begin{lemma}
	\label{lem:K_prop}
	Let $A$ be the generator of an immediately differentiable
	$C_0$-semigroup $(T(t))_{t \geq 0}$ of quasi-multiplication type
	on a Banach space $X$.
	Let $b \geq 0$ satisfy 
	$\sigma(A) \cap i \mathbb{R} \subset 
	(-ib,ib)$, and 
	define $M \colon [b,\infty) \to (0,\infty)$ by
	\eqref{eq:M_def_multiplication}.
	If $M(s) = o(s)$  as $s \to \infty$,
	then the function $K_{\varepsilon}$ defined 
	by \eqref{eq:K_eps_def} satisfies 
	the following properties for all $\varepsilon \in (0,1)$:
	\begin{enumerate}
		\renewcommand{\labelenumi}{\rm{\alph{enumi})}}
		\item 
		There exists a constant $s_0 \geq b$ such that 
		\begin{equation}
			\label{eq:se_M_K_bounded}
			s e^{- M(s)/t} \leq K_{\varepsilon}(t)
		\end{equation}
		for all $s \geq s_0$ and $t >0$. In particular,
		$K_{\varepsilon}(t) >0$ for all $t >0$.
		\item
		$K_{\varepsilon}(t) \to \infty$ as $t \to \infty$, and 
		there exists a constant 
		$t_0 >0$ such that 
		$K_{\varepsilon}$ is non-decreasing and left-continuous on 
		the interval $[t_0,\infty)$.
	\end{enumerate}
\end{lemma}
\begin{proof}
	Let $0 < \varepsilon < 1$.
	Since $M(s) = o(s)$  as 
	$s \to \infty$,
	there exists $s_1 \geq b$ such that 
	\begin{equation}
		\label{eq:wN_bound}
		M(s) \leq \varepsilon s
	\end{equation}
	for all $s \geq s_1$.
	
	a)
	By Theorem~\ref{thm:spectral_prop},
	there exists $s_2 \geq b$ such that 
	\[
	\sup\{
	\re z :
	z \in \sigma(A) \text{~and~} |z| \geq (1-\varepsilon)s_2
	\} < 0.
	\]
	Define $s_0 \coloneqq \max \{s_1,\,s_2 \}$.
	Since 
	\[
	\inf_{|\eta| \geq s} \dist(i \eta , \sigma(A)) 
	= M(s) \to \infty
	\qquad \text{as $s \to \infty$},
	\]
	Lemma~\ref{lem:inf_dist_point}  and the argument
	used to obtain \eqref{eq:alp_beta_bound}
	show that
	for all $s \geq s_0$,
	there exist $u_s,v_s \in \mathbb{R}$
	with $u_s +iv_s  \in \sigma(A)$ such that
	\begin{subequations}
		\label{eq:alp_beta_bound_mult}
		\begin{empheq}[left = {\empheqlbrace \,}, right = {}]{align}
			|u_s | &\leq M(s), \label{eq:alp_beta_bound_mult1}\\
			(1-\varepsilon)s &\leq |u_s  + iv_s | , \quad \text{and}
			\label{eq:alp_beta_bound_mult2} \\
			u_s &< 0 \label{eq:alp_beta_bound_mult3}.
		\end{empheq}
	\end{subequations}
	
	Let $s \geq s_0$ and $t>0$.
	By \eqref{eq:alp_beta_bound_mult1}--\eqref{eq:alp_beta_bound_mult3},
	\begin{equation}
		\label{eq:1-eps_se}
		(1-\varepsilon) se^{-M(s)/t} \leq 
		|u_s+iv_s| se^{u_s/t}.
	\end{equation}
	Using \eqref{eq:multi_AT},
	we also have
	\begin{equation}
		\label{eq:usvs_e}
		|u_s + i v_s| e^{u_s/t} \leq 
		\sup_{z \in \sigma(A)} |z| e^{(\re z)/t} = 
		\left\|AT\left( \frac{1}{t} \right) \right\|
		=(1-\varepsilon) K_{\varepsilon}(t).
	\end{equation}
	The estimate \eqref{eq:se_M_K_bounded}
	follows from \eqref{eq:1-eps_se} and \eqref{eq:usvs_e}.
	
	b)
	By \eqref{eq:se_M_K_bounded} and \eqref{eq:wN_bound},
	\[
	\sup_{s \geq s_0} s e^{-\varepsilon s/t} \leq K_{\varepsilon}(t)
	\]
	for all $t >0$.
	Hence, $K_{\varepsilon}(t) \to \infty$ as $t \to \infty$.

	Since $\sigma(A) \cap \overline{\mathbb{C}}_+$ is compact
	by Theorem~\ref{thm:spectral_prop}, 
	there exists $t_0 \geq 1$ such that for all $t \geq t_0$,
	\[
	\sup_{z \in \sigma(A) \cap \overline{\mathbb{C}}_+} |z| e^{(\re z)/t}
	\leq  \sup_{z \in \sigma(A) \cap \overline{\mathbb{C}}_+} |z| e^{\re z}
	< (1-\varepsilon) K_{\varepsilon}(t).
	\]
	Therefore, by \eqref{eq:usvs_e},
	\begin{equation}
		\label{eq:K_eps_left_p}
		(1-\varepsilon)K_{\varepsilon}(t) = \sup_{z \in \sigma(A) \cap \mathbb{C}_-} |z| e^{(\re z)/t}
	\end{equation}
	for all $t \geq t_0$.
	This implies that $K_{\varepsilon}$
	is non-decreasing on $[t_0,\infty)$.
	
	It remains to show that $K_{\varepsilon}$ is 
	left-continuous on $[t_0,\infty)$.
	We
	let $t > t_0$ and $\delta >0$. By
	\eqref{eq:K_eps_left_p},
	there exists $z_1 \in \sigma(A) \cap \mathbb{C}_-$ such that 
	\begin{equation}
		\label{eq:K_eps_z_1_upper}
		K_{\varepsilon}(t) \leq \frac{|z_1| e^{(\re z_1)/t}}{1-\varepsilon} + \frac{\delta}{2}.
	\end{equation}
	Moreover, there exists $h_0 \in (0,t-t_0]$ such that 
	for all $h \in (0,h_0)$,
	\begin{equation}
		\label{eq:cont_prop_z1}
		\frac{|z_1| e^{(\re z_1)/t}}{1-\varepsilon}  - 
		\frac{|z_1| e^{(\re z_1)/(t-h)}}{1-\varepsilon} \leq 
		\frac{\delta}{2}.
	\end{equation}
	Since $K_{\varepsilon}$
	is non-decreasing on $[t_0,\infty)$,
	it follows that 
	\begin{equation}
		\label{eq:K_th_Kt}
		K_{\varepsilon}(t-h) \leq K_{\varepsilon}(t)
	\end{equation}
	for all $h \in (0,h_0)$.
	Combining \eqref{eq:K_eps_z_1_upper}--\eqref{eq:K_th_Kt}, 
	we derive
	\begin{align*}
		\frac{|z_1| e^{(\re z_1)/t}}{1-\varepsilon} - \frac{\delta}{2} 
		& \leq 
		\frac{|z_1| e^{(\re z_1)/(t-h)}}{1-\varepsilon}\\ &\leq 
		K_{\varepsilon}(t-h)  \leq K_{\varepsilon}(t) 
		\leq 
		\frac{|z_1| e^{(\re z_1)/t}}{1-\varepsilon} + 
		\frac{\delta}{2}
	\end{align*}
	for all $h \in (0,h_0)$.
	Hence,
	$
	|K_{\varepsilon}(t) - K_{\varepsilon}(t-h)| \leq \delta
	$
	for all $h \in (0,h_0)$.
	Since $\delta >0$ is arbitrary, 
	$K_\varepsilon$ is left-continuous on $[t_0,\infty)$.
\end{proof}

For the proof of the upper estimate in statement b)
of Theorem~\ref{thm:mul_lower_bound},
we introduce right-inverses.
Given a constant 
$b \geq 0$ and a non-decreasing function
$M \colon [b,\infty) \to \mathbb{R}_+$ such that $M(s) \to \infty$ as $s \to \infty$,
the {\em right-inverse} $M_r^{-1}\colon [M(b),\infty) \to [b,\infty)$
of $M$ is defined by
\[
M_r^{-1}(t) \coloneqq \sup\{
s \geq b \colon M(s) \leq t
\}.
\]
Right-inverses have properties similar to 
those presented for left-inverses in Section~\ref{sec:left_inv}.
Indeed, $M_r^{-1}$ is non-decreasing and right-continuous.
Moreover, $M_r^{-1}(t) \to \infty$ as $t \to \infty$, and 
\begin{equation}
	\label{eq:ri_prop1}
	M_r^{-1}(M(s)) \geq s
\end{equation}
for all $s \geq b$. 
If $M$ is continuous, then
\begin{equation}
	\label{eq:ri_prop2}
	M(M_r^{-1}(t)) = t
\end{equation}
for all $t \geq M(b)$.
A property analogous to that established for left-inverses
in Lemma~\ref{lem:left_inv_small_o} also holds for right-inverses. Its proof is the same as that of Lemma~\ref{lem:left_inv_small_o} and hence is omitted.
\begin{lemma}
	\label{lem:right_inv_small_o}
	Let $b\geq 0$ and
	let $M\colon [b,\infty) \to (0,\infty)$ be a non-decreasing  continuous
	function such that $M(s) \to \infty$ as $s \to \infty$.
	If $M(s) = o(s)$ as $s \to \infty$,
	then
	$t = o(M_r^{-1}(t))$
	as $t \to \infty$.
\end{lemma}

We are now ready to prove Theorem~\ref{thm:mul_lower_bound}.
\begin{proof}{Proof of Theorem~\ref{thm:mul_lower_bound}}
	Let $0 < \varepsilon < 1$ and define $K_{\varepsilon}$
	by \eqref{eq:K_eps_def}.
	We immediately see that the function
	$M$ defined by \eqref{eq:M_def_multiplication} is 
	non-decreasing and continuous.
	Moreover, 
	$M(s) \to \infty$ as $s \to \infty$;
	see the paragraph preceding Lemma~\ref{lem:K_prop}.
	These properties of $M$ will be used for 
	the right-inverse $M_r^{-1}$.
	
	a)
	By Lemma~\ref{lem:K_prop}.a), there exists $s_0 \geq b$ such that 
	\begin{equation}
		\label{eq:M_t_lower}
		\frac{M(s)}{t} \geq  \log \left( \frac{s}{K_{\varepsilon}(t)} \right)
	\end{equation}
	for all $s \geq s_0$ and $t>0$.
	By Lemma~\ref{lem:K_prop}.b), $K_{\varepsilon}(t) \to \infty$ 
	as $t \to \infty$, and
	there exists $t_1>0$ such that 
	$K_{\varepsilon}$ 
	is non-decreasing and left-continuous on the interval $[t_1,\infty)$. 
	We regard $K_{\varepsilon}$ as a function on $[t_1,\infty)$ and 
	consider the left-inverse $K_{\varepsilon}^{-1}\colon
	[K_{\varepsilon}(t_1),\infty) \to [t_1,\infty)$.
	Then we have
	$K_{\varepsilon}(K_{\varepsilon}^{-1}(s)) \leq s$ for all $s \geq  K_{\varepsilon}(t_1)$
	by the property \eqref{eq:li_prop3}	of left-inverses.
	Let $\lambda > 1$ and $s_1 \coloneqq 
	\max \{s_0,\,K_{\varepsilon}(t_1)\}$. 
	For all $s \geq s_1$,
	the estimate
	\eqref{eq:M_t_lower} with $t = 
	K_{\varepsilon}^{-1}(s)$ gives
	\[
	\frac{M(\lambda s)}{K_{\varepsilon}^{-1}(s)} \geq 
	\log \left( \frac{\lambda s}{K_{\varepsilon}(K_{\varepsilon}^{-1}(s))}  \right) 
	\geq
	\log \lambda.
	\]
	Since $\lambda > 1$ is arbitrary,
	it follows that 
	\begin{equation}
		\label{eq:Minf_Kinv}
		M_{\inf}(s) = \inf_{\lambda > 1}
		\frac{M(\lambda s)}{\log \lambda}
		\geq K_{\varepsilon}^{-1}(s)
	\end{equation}
	for all $s \geq s_1$.
	Combining this with $K_{\varepsilon}^{-1}(s) \to \infty$ as $s \to \infty$,
	we obtain
	$M_{\inf}(s) \to \infty$ as $s \to \infty$.
	Hence, Lemma~\ref{lem:M_inf_prop1}
	shows that statement~a) holds.
	
	b)
	The lower  estimate in statement b) follows immediately 
	from \eqref{eq:Minf_Kinv}.
	Indeed, \eqref{eq:Minf_Kinv} yields
	\[
	K_{\varepsilon}(t) = (K_{\varepsilon}^{-1})^{-1}(t) \geq M_{\inf}^{-1}(t)
	\]
	for all $t > M_{\inf}(s_1)$.
	Therefore,
	by the definition \eqref{eq:K_eps_def} of $K_{\varepsilon}$,
	\[
	\|AT(t)\| \geq (1-\varepsilon )M_{\inf}^{-1}\left(
	\frac{1}{t} \right)
	\]
	for all $t \in (0,1/M_{\inf}(s_1))$.
	
	Next, we prove the upper estimate in statement b).
	Take $z \in \sigma(A)$ such that $|\im z| \geq b$.
	We have
	\[
	M(|\im z|) =
	\inf_{|\eta| \geq |\im z|} \dist (i \eta, \sigma(A)) \leq |i\im z - z| = |\re z|.
	\]
	By \eqref{eq:ri_prop1}, the right-inverse $M_r^{-1}$ satisfies
	\begin{equation}
		\label{eq:imz_Mrinv_bound}
		|\im z| \leq M_r^{-1}(M(|\im z|)) \leq M_r^{-1}(|\re z|).
	\end{equation}
	Since 
	$M(s) = o(s)$ as $s \to \infty$ by assumption,
	Lemma~\ref{lem:right_inv_small_o} shows that
	there exists $\tau _1 \geq M(b)$ 
	such that $\tau  \leq \varepsilon M_r^{-1}(\tau )$
	for all $\tau  \geq \tau _1$.
	If $|\re z| \geq \tau _1$, then \eqref{eq:imz_Mrinv_bound} gives
	\begin{align}
		\label{eq:z_Mrinv_bound}
		|z| 
		\leq 
		(|\re z|^2 + M_r^{-1}(|\re z|)^2)^{1/2} 
		\leq 
		(\varepsilon^2+1)^{1/2}  M_r^{-1}(|\re z|).
	\end{align}
	
	Define the subsets $\Omega_0$, $\Omega_1$, and 
	$\Omega_2$ of $\sigma(A)$ by
	\begin{align*}
		\Omega_0  &\coloneqq \{z \in \sigma(A) :\re z \leq -\tau_1 
		\text{~and~} |\im z| \geq b \}, \\
		\Omega_1 &\coloneqq \{ z \in \sigma(A) : \re z > -\tau _1 \}, \\
		\Omega_2 &\coloneqq \{ z \in \sigma(A) : \re z \leq -\tau _1
		\text{~and~} |\im z| < b\}.
	\end{align*}
	Then $\sigma(A)=  \Omega_0 \cup \Omega_1 \cup \Omega_2$.
	By the estimate \eqref{eq:z_Mrinv_bound} and 
	the property \eqref{eq:ri_prop2}
	of right-inverses,
	\begin{align}
		\sup_{z \in \Omega_0} |z| e^{t \re z} &\leq 
		(\varepsilon^2+1)^{1/2}  \sup_{z \in \Omega_0} M_r^{-1}(|\re z|) e^{-t |\re z|} \notag \\
		&\leq 
		(\varepsilon^2+1)^{1/2} 
		\sup_{s \geq b} s e^{-t M(s)}. 
		\label{eq:Omega_bound}
	\end{align}
	Since $ \Omega_1$ 
	is bounded by Theorem~\ref{thm:spectral_prop},
	there exists  $C_1>0$ such that
	\begin{equation}
		\label{eq:Omega1_bound}
		\sup_{z \in \Omega_1} |z| e^{t \re z}
		\leq C_1
	\end{equation}
	for all $t \in (0,1]$.
	For all $z \in \Omega_2$,
	\begin{align*}
		|z| e^{t \re z} &\leq (|\re z|^2 + b^2)^{1/2} e^{t \re z} \\
		&\leq 
		C_2 |\re z|e^{t \re z},\quad  \text{where~}C_2 \coloneqq 
		\left(
		1 + \frac{b^2}{\tau _1^2}
		\right)^{1/2}.
	\end{align*}
	This implies that
	\begin{equation}
		\label{eq:Omega2_bound}
		\sup_{z \in  \Omega_2} |z| e^{t \re z}
		\leq 
		C_2 \sup_{s \geq 0}  se^{-t s}
		= \frac{C_2}{et}
	\end{equation}
	for all $t > 0$.
	Combining the estimates \eqref{eq:Omega_bound}--\eqref{eq:Omega2_bound},
	we derive
	\begin{equation}
		\label{eq:Omega_bound_max}
		\sup_{z \in  \sigma(A)} |z| e^{t \re z} \leq 
		\max
		\left\{
		(\varepsilon^2+1)^{1/2} 
		\sup_{s \geq b} s e^{-t M(s)},\,
		C_1,\,
		\frac{C_2}{et}
		\right\}
	\end{equation}
	for all $t \in (0,1]$.
	
	It is enough to show that 
	\begin{equation}
		\label{eq:sexp_Minf_inv}
		\sup_{s \geq b} s e^{-t M(s)} \leq M_{\inf}^{-1}\left(
		\frac{1}{t}
		\right)
	\end{equation}
	for all $t \in (0,1/M_{\inf}(b)]$. Indeed,
	since 
	\[
	M_{\inf} (s) =	\inf_{\lambda >1} \frac{M(\lambda s)}{\log \lambda} \leq \frac{M(es)}{\log e} = M(es)
	\]
	for all $s \geq b$, it follows that 
	$M_{\inf}(s) = o(s)$ as $s \to \infty$.
	Lemma~\ref{lem:left_inv_small_o} shows that
	$1/t = o(M_{\inf}^{-1}(1/t))$ as $t \downarrow 0$.
	Therefore, if \eqref{eq:sexp_Minf_inv} is true, 
	then by \eqref{eq:Omega_bound_max},  there exists $t_0 >0$ such that 
	\[
	\|AT(t)\| \leq (1+\varepsilon)M_{\inf}^{-1}\left(
	\frac{1}{t}
	\right)
	\]
	for all $t \in (0,t_0]$.
	
	Let $0 < t \leq 1/M_{\inf}(b)$, and define
	\[
	R \coloneqq M_{\inf}^{-1}\left(
	\frac{1}{t}
	\right).
	\]
	If $b \leq s  \leq R$, then
	\[
	s e^{-tM(s)} \leq s \leq R.
	\]
	Let $s > R$.
	Setting $\lambda = s/R > 1$, we obtain
	\begin{equation}
		\label{eq:Minf_Mlog_bound}
		M_{\inf}(R) = 
		\inf_{\lambda >1} \frac{M(\lambda R)}{\log \lambda} \leq \frac{M(s)}{\log (s/R)}.
	\end{equation}
	Since $M_{\inf}(R) = 1/t$ by the property \eqref{eq:li_prop2}
	of left-inverses, the estimate \eqref{eq:Minf_Mlog_bound} gives
	\[
	-tM(s) = -\frac{M(s)}{M_{\inf}(R)} \leq -\log
	\left( \frac{s}{R} \right)= \log \left( \frac{R}{s} \right).
	\]
	This implies that
	\[
	s e^{-tM(s)}  \leq s e^{\log(R/s)} = R.
	\]
	Thus, the desired estimate \eqref{eq:sexp_Minf_inv} holds 
	for all $t \in (0,1/M_{\inf}(b)]$.
\end{proof}

\printbibliography
\end{document}